\documentclass[11pt,reqno]{amsart}
\usepackage{amsmath,amssymb,amsthm,mathtools}
\usepackage[a4paper,margin=1in]{geometry}
\usepackage{microtype,xurl,needspace}
\usepackage[hidelinks]{hyperref}
\numberwithin{equation}{section}
\newtheorem{theorem}{Theorem}[section]
\newtheorem{lemma}[theorem]{Lemma}
\newtheorem{proposition}[theorem]{Proposition}

\newtheorem{example}[theorem]{Example}
\theoremstyle{remark}
\newtheorem{remark}[theorem]{Remark}
\newcommand{\dV}{\,\mathrm dV_g}
\newcommand{\Div}{\operatorname{div}_g}
\newcommand{\G}{\mathcal G}
\newcommand{\F}{\mathcal F_\rho}
\newcommand{\D}{\mathcal D}
\newcommand{\Hh}{\mathcal H}
\newcommand{\Kk}{\mathcal K}
\newcommand{\Rr}{\mathcal R}
\newcommand{\norm}[1]{\left\|#1\right\|}
\title[Subcritical heat flow for Liouville systems]{Global existence and convergence of a density heat flow for strictly subcritical Liouville systems}
\author{Zetao Cheng}
\address{Department of Mathematics and Research Institute for Natural Sciences, College of Natural Sciences, Hanyang University, 222 Wangsimni-ro Seongdong-gu, Seoul 04763, Republic of Korea}
\email{chengzt20@hanyang.ac.kr}
\author{Haoyu Li}
\address{Department of Mathematics, Faculty of Science and Technology, University of Macau, Taipa, Macau, China}
\email{hyli1994@hotmail.com}
\author{Woongbae Park}
\address{Department of Mathematics Education, Korea National University of Education, Cheongju-si, Chungbuk, South Korea}
\email{ubmpark5@knue.ac.kr}
\author{Lei Zhang}
\address{Department of Mathematics, University of Florida, Gainesville, FL 32611-8105, USA}
\email{leizhang@ufl.edu}
\date{}
\usepackage{xcolor}
\begin{document}
\begin{abstract}
We study a density heat flow associated with normalized Liouville
systems on a closed connected Riemannian surface, and we show that this flow leads to the existence of solutions for the corresponding subcritical Liouville systems. The coupling matrix is symmetric and entrywise nonnegative, with no invertibility, irreducibility or inverse-matrix sign assumption. Under the strict subcriticality conditions $8\pi\sum_{i\in J}\rho_i>\sum_{i,j\in J}a_{ij}\rho_i\rho_j$ for every nonempty set of components $J$, every smooth positive initial density of masses $\rho_i>0$ generates a unique global smooth positive solution converging in every $C^m$ norm to a stationary density. The limiting potentials solve the stationary Liouville system. Our proof uses entropy coercivity and a gradient inequality in the density variables, which also yields algebraic or exponential convergence rates.
\end{abstract}
\maketitle

\section{Introduction}
Let $(M,g)$ be a smooth closed connected Riemannian surface of area one. For $I=\{1,\ldots,n\}$ with $n\ge1$, smooth positive weights $h_i$ and parameters $\rho_i>0$, consider the normalized Liouville system
\[
 {\Delta_g u_i+\sum_{j\in I}a_{ij}\rho_j
 \left(\frac{h_j e^{u_j}}{\int_M h_j e^{u_j}\dV}-1\right)=0}{\text{ in }M}{,
 \qquad \int_M u_i\dV=0,\qquad i\in I.}
\]
Here $A=(a_{ij})$ is the coupling matrix. We write $dV_g$ for the
Riemannian area element, $\nabla_g$ for the gradient with respect to the metric $g$ and $\Div$ for the  divergence, with $\Delta_g=\Div\nabla_g$. For a smooth function $f$ and a vector field $X$, these conventions mean $g(\nabla_g f,X)=df(X)$ and
$\int_M f\Div X\dV=-\int_M g(\nabla_g f,X)\dV$. Integrals without a specified domain are over $M$ with respect to $dV_g$.
The associated densities
\[
 c_i=\frac{\rho_i h_i e^{u_i}}{\int_M h_i e^{u_i}\dV}
 \qquad\text{satisfy}\qquad \int_M c_i\dV=\rho_i.
\]
They allow us to construct a mass-preserving flow whose positive stationary states recover the Liouville system.

We write $d_g$ for the geodesic distance and $B_r(x)$ for the geodesic ball with radius $r$ and centered at $x\in M$. Let $G$ be the symmetric Green function normalized by
\begin{equation}\label{eq:green}
-\Delta_{g,x}G(x,y)=\delta_y-1,\qquad \int_M G(x,y)\dV(x)=0.
\end{equation}
The subscript $x$ means that the Laplace--Beltrami operator acts on the first variable. The equation holds in the sense of distributions, with $\delta_y$ the unit Dirac mass at $y$.
The area normalization makes the right-hand side have integral zero. For densities of masses $\rho_i$, we define
\begin{equation}\label{eq:potential}
U_i[c]=\sum_{j=1}^n a_{ij}\G(c_j-\rho_j),\qquad
\G f(x)=\int_M G(x,y)f(y)\dV(y).
\end{equation}
Then
\begin{equation}\label{eq:poisson}
-\Delta_g U_i[c]=\sum_j a_{ij}(c_j-\rho_j),\qquad \int_M U_i[c]\dV=0.
\end{equation}
In follows, we refer (\ref{eq:poisson}) as zero-mean Liouville system. For a zero-mean Liouville solution, uniqueness of the zero-mean Poisson solution gives $u_i=U_i[c]$. Thus $\log(c_i/h_i)-U_i[c]$ is constant on $M$. This relation leads to the density heat flow
\begin{equation}\label{eq:flow}
{\partial_t c_i=\Div\left(c_i\nabla_g\left(\log\frac{c_i}{h_i}-U_i[c]\right)\right)}{\text{ for }t>0,\,x\in M\text{ and }
i\in I.}
\end{equation}
Integration over the closed surface preserves each mass. Lemma~\ref{lem:stationary} proves the converse correspondence between positive stationary densities and zero-mean Liouville solutions. We prove global existence and convergence of this flow in the strictly
subcritical region defined later.

The coupling matrix satisfies
\begin{equation}\label{hyp:A}
A=A^T,\qquad a_{ij}\ge0\quad(i,j\in I).
\end{equation}
For a nonempty $J\subseteq I$, the subsystem indexed by $J$ has coupling matrix $A_J=(a_{ij})_{i,j\in J}$ and mass vector $(\rho_i)_{i\in J}$. This restriction of indices need not decouple from the remaining components. Its strict subcriticality is expressed by $\Lambda_J(\rho)>0$, where
\begin{equation}\label{hyp:subcritical}
\Lambda_J(\rho)=8\pi\sum_{i\in J}\rho_i-\sum_{i,j\in J}a_{ij}\rho_i\rho_j.
\end{equation}
We assume $\Lambda_J(\rho)>0$ for all such $J$, including $J=I$.

The convergence and spatial regularity in the main theorem are measured by $\|f\|_{C^m}=\sum_{j=0}^m\|\nabla_g^j f\|_\infty$. For $0<\alpha<1$, we set $\|f\|_{C^{m,\alpha}}=\|f\|_{C^m}+[\nabla_g^m f]_\alpha$. The seminorm compares tensors by parallel transport along the unique minimizing geodesic for $0<d_g(x,y)<r_0$, where ${0<r_0}{<\operatorname{inj}(M,g)}$ is fixed. Here $\operatorname{inj}(M,g)>0$ denotes the injectivity radius of $(M,g)$. We set
\[ [T]_\alpha=\sup_{0<d_g(x,y)<r_0}
 \frac{|T(x)-\mathsf P_{y\to x}T(y)|_g}{d_g(x,y)^\alpha}.
\]
These norms are equivalent to the norms defined using a fixed finite atlas. Here $m$ is a nonnegative integer, $\|\cdot\|_\infty$ is the essential supremum and $\nabla_g^j$ denotes the $j$th covariant derivative. For vectors, these norms are the maximum of the component norms.

\begin{theorem}\label{thm:main}
Suppose that \eqref{hyp:A} holds and $\Lambda_J(\rho)>0$ for every nonempty $J\subseteq I$. For any $c_{i,0}\in C^\infty(M)$ with $c_{i,0}>0$ and $\int_M c_{i,0}\dV=\rho_i$, system \eqref{eq:flow} with $c_i(0)=c_{i,0}$ has a unique global smooth positive solution preserving each mass. Moreover, there exists a smooth positive stationary state $c^\infty$ such that, for every integer $m\ge0$,
\begin{equation}\label{eq:whole}
\lim_{t\to\infty}\norm{c(t)-c^\infty}_{C^m(M)^n}=0.
\end{equation}
For every $\tau>0$, integer $m\ge0$ and $0<\alpha<1$, we get
\begin{equation}\label{eq:high}
\begin{split}
&\sup_{t\ge\tau}\sum_i\left(\norm{c_i(t)}_{C^{m,\alpha}}+\norm{U_i[c(t)]}_{C^{m+2,\alpha}}\right)<\infty,\\
&\inf_{t\ge\tau,\,x\in M,\,i\in I}c_i(t,x)>0.
\end{split}
\end{equation}
The zero-mean functions $u_i^\infty=U_i[c^\infty]$ satisfy
\begin{equation}\label{eq:elliptic}
\Delta_g u_i^\infty+\sum_j a_{ij}\rho_j
\left(\frac{h_j e^{u_j^\infty}}{\int_M h_j e^{u_j^\infty}\dV}-1\right)=0\text{ in }M\text{ for }i\in I.
\end{equation}
\end{theorem}
\begin{remark}
One should note the following details.
\begin{itemize}
    \item [$(a)$] The limit $u_i^\infty$ may depend on the initial data. Uniqueness in the theorem concerns the evolution with prescribed initial data.
    \item [$(b)$] Since $u_i^\infty$ is smooth on the compact surface, each component has finite Dirichlet energy and a finite positive exponential integral. More precisely, we get
 $$\int_M|\nabla_g u_i^\infty|_g^2\dV<\infty\text{ and }0<(\min_M h_i)e^{\min_M u_i^\infty}
 \le\int_M h_i e^{u_i^\infty}\dV \le(\max_M h_i)e^{\max_M u_i^\infty}<\infty.$$
Thus the densities $\rho_i h_i e^{u_i^\infty}/\int_M h_i e^{u_i^\infty}\dV$ are well defined and have masses $\rho_i$.
\end{itemize}

\end{remark}

The coefficient assumptions can be compared with those of Lin and Zhang. Their entire-solution uniqueness and local bubbling estimates \cite[(1.2) and Theorems 1.1--1.2]{LZprofile} use
\[
 (H_1)\qquad A\text{ is symmetric, entrywise nonnegative, irreducible and invertible}.
\]
For a symmetric matrix, irreducibility means that no nonempty proper $J\subset I$ satisfies $a_{ij}=0$ for all $i\in J$ and $j\notin J$. Their topological degree results \cite[pp.~558--559]{LZdegree} also assume, with $A^{-1}=(a^{ij})$,
\[
 (H_2)\qquad a^{ii}\le0,\qquad a^{ij}\ge0\ (i\ne j), \qquad \sum_{j\in I}a^{ij}\ge0\quad(i\in I).
\]
Their one-bubble analysis at critical parameters \cite[p.~2587]{LZone} assumes $(H_1)$ alone. Theorem~\ref{thm:main} requires only symmetry and entrywise nonnegativity of $A$, together with strict subcriticality for every nonempty subsystem. It therefore includes singular and reducible matrices as well as indefinite ones. The proof works with the density free energy and never inverts $A$. These assumptions yield global existence and convergence of the density heat flow at fixed strictly subcritical masses, whereas the cited results concern elliptic bubbling and topological degree. For the matrices considered here, the full family $\Lambda_J(\rho)>0$ cannot be replaced by $\Lambda_I(\rho)>0$ alone, as the example in Section~\ref{s:Complementary} shows.

On $\mathbb R^2$, Karmakar and Wolansky \cite[Definition 1.1 and Theorem 1.3]{KW} proved global existence and uniqueness of weak solutions for symmetric nonnegative coupling under the same strict subsystem conditions, starting from densities with finite entropy and second moment. Their solutions satisfy a free-energy dissipation inequality. Theorem~\ref{thm:main} concerns a closed surface with smooth positive weights and proves uniform regularity, strict positivity and convergence of the whole trajectory. Its entropy estimate uses the system logarithmic Hardy--Littlewood--Sobolev inequality of Shafrir and Wolansky \cite{SW}.

For scalar equations on a flat torus, Choi, Jeong and Seo \cite[Theorems 1.2 and 6.2]{CJS} proved smooth convergence of uniformly bounded solutions under their assumptions on an analytic external potential and the interaction kernel, including a Keller--Segel
application. Here we treat several coupled densities on a general closed connected surface with smooth positive weights. The strict subsystem conditions provide the uniform bounds needed for convergence. We establish the gradient inequality directly in the density variable, using the local analytic expansion of the logarithm around a positive stationary density. The coefficients of this power series are smooth functions on $M$. The analyticity concerns the density variable in a H\"older space. To state the quantitative convergence in Sobolev norms, we use the spectral definition. If $-\Delta_g e_\ell=\lambda_\ell e_\ell$ is an orthonormal eigenbasis, then for $s\in\mathbb R$
\[ \|f\|_{H^s}^2=\sum_\ell(1+\lambda_\ell)^s |\langle f,e_\ell\rangle|^2,\qquad H^{-1}(M)=(H^1(M))^*. \]
For vectors, $L^2$, $H^s$ and $H^{-1}$ use the Hilbert product norms.

The supplementary Theorem~\ref{thm:rates} gives rates for the energy, the $H^{-1}$ distance and the mass-weighted Wasserstein distance. Interpolation and elliptic regularity give corresponding rates for the densities and potentials in higher spatial norms.

For the transport estimate, we normalize each component to a probability measure. Let $\mathcal P(M)$ denote the Borel probability measures on $M$. For $\mu,\nu\in\mathcal P(M)$, a coupling is a probability measure $\pi\in\mathcal P(M\times M)$ whose first and second marginals are $\mu$ and $\nu$. Explicitly, this means that
\[
 \pi(E\times M)=\mu(E),\qquad \pi(M\times E)=\nu(E)
 \quad\text{for every Borel set }E\subseteq M.
\]
Writing $\Pi(\mu,\nu)$ for all such couplings, we define
\[
 W_2(\mu,\nu)^2=\inf_{\pi\in\Pi(\mu,\nu)}
 \int_{M\times M}d_g(x,y)^2\,d\pi(x,y).
\]
If $X$ is a measurable map, its pushforward is $X_\#\mu(E)=\mu(X^{-1}(E))$ for Borel sets $E$.
The measures $p_i\,dV_g$, where $p_i=c_i/\rho_i$, are probabilities. On densities with component masses $\rho_i$, we use 
\begin{equation}\label{eq:metric}
d_\rho(c,\widetilde c)^2
=\sum_i\rho_i W_2^2\left(\frac{c_i}{\rho_i}\,\mathrm dV_g,
\frac{\widetilde c_i}{\rho_i}\,\mathrm dV_g\right).
\end{equation}
For the definition of $W_2$ and the continuity-equation estimates used below, see \cite[Sections~2.2 and~2.5]{AS}. Remark~\ref{rem:metric} explains the factors $\rho_i$
in \eqref{eq:metric}.

The estimates leading to global existence use the spatial norms, for $k\ge0$ an integer,
\[
 \|f\|_q=\left(\int_M|f|^q\dV\right)^{1/q}\quad(1\le q<\infty),
 \qquad \|f\|_\infty=\operatorname*{ess\,sup}_{M}|f|,
 \qquad \|f\|_{W^{k,q}}=\sum_{j=0}^k\|\nabla_g^j f\|_q.
\]
These norms extend to covariant tensors using the pointwise metric norm. For vectors, $L^2$ uses the Hilbert product norm and the other $L^q$ and $W^{k,q}$ norms use the maximum of the component norms.

Throughout this paper, the constants may depend on the fixed data $(M,g),A,\rho,h,c_0$ and on specified regularity indices or positive times. The mass vector $\rho$ is fixed in the strictly subcritical region. Constants may diverge as $\rho$ approaches the boundary of the set $\{\rho\in\mathbb{R}_+^n|\Lambda_J(\rho)>0\mbox{ for any }J\subset I\}$. The letter $C$ may change from line to line unless explicitly fixed.

The energy bound alone gives no pointwise control of the densities. We first use entropy coercivity to make the mass on small balls uniformly small and then absorb the coupling terms in localized $L^p$ estimates. The resulting uniform regularity gives a positive stationary subsequential limit. A first-exit argument keeps the trajectory in the neighborhood where the gradient inequality applies and yields convergence of the whole trajectory.

Section~\ref{sec:energy} gives the free energy and characterizes stationary states. Section~\ref{sec:local} proves local existence and continuation. Section~\ref{sec:global} establishes the uniform bounds and global existence. Section~\ref{sec:convergence} proves whole-trajectory convergence. Section~\ref{s:Complementary} provides the supplementary rate and explains the role of the subsystem conditions.

\section{Free energy and stationary states}\label{sec:energy}
The Green function has the form
\begin{equation}\label{eq:singularity}
G(x,y)=-\frac1{2\pi}\log d_g(x,y)+R(x,y),\qquad \norm R_\infty<\infty.
\end{equation}
We use the boundedness of $R$ in the interaction estimates below. For nonnegative densities of finite entropy and prescribed masses, we define \emph{the free-energy functional}
\begin{equation}\label{eq:energy}
\begin{aligned}
\F(c)
&:=\overbrace{H(c)}^{\text{Boltzmann entropy}}
+\overbrace{\bigl(-W_A(c)\bigr)}^{\text{interaction energy}}
+\overbrace{\left(-\sum_i\int c_i\log(\rho_i h_i)\dV\right)}
^{\text{external-potential energy}},\\
H(c)
&=\sum_i\int c_i\log c_i\dV,\\
W_A(c)
&=\frac12\sum_{i,j}a_{ij}
\int(c_i-\rho_i)\G(c_j-\rho_j)\dV.
\end{aligned}
\end{equation}
We use the convention $0\log0=0$.

\begin{lemma}\label{lem:energy-finite}
For every $i\in I$, suppose that $c_i\ge0$ has mass $\rho_i$ and finite entropy $\int_M c_i\log c_i\dV$. Then the integrals defining the interaction energy $-W_A(c)$ in \eqref{eq:energy} are absolutely convergent, and $\F(c)$ is a finite real number.
\end{lemma}
\begin{proof}
Put $p_i=c_i/\rho_i$. Since $\operatorname{vol}_g(M)=1$, both $dV_g$ and $p_i\,dV_g$ are probability measures. For a bounded measurable function $\psi$, the entropy inequality
\cite[Eq.~(2.8)]{BOPR} gives
\[
\int p_i\psi\dV\le\int p_i\log p_i\dV+\log\int e^\psi\dV.
\]
Indeed, integrate $s\log s-s+1\ge0$ against the probability density $q=e^\psi/\int e^\psi\dV$ with $s=p_i/q$. Fix $0<\gamma<2$ and apply the inequality to $\psi_N(y)=\min\{N,\gamma\log^+(1/d_g(x,y))\}$. The small-ball volume bound on the surface gives
\[
\sup_x\int_M\max\{1,d_g(x,y)^{-\gamma}\}\dV(y)<\infty.
\]
For example, summing over annuli of radii $2^{-k}r_0$ bounds the singular part by a constant times $\sum_{k\ge0}(2^{-k}r_0)^{2-\gamma}$. Monotone convergence therefore gives
\[
\sup_x\int p_j(y)\log^+(1/d_g(x,y))\dV(y)<\infty.
\]
The negative part of $\log(1/d_g(x,y))$ is bounded by $\log^+\operatorname{diam}(M)$, and $R$ is bounded. Applying the same kernel estimate also to the constant density gives
\[
\sup_x\int_M |G(x,y)|c_j(y)\dV(y)<\infty,
\qquad \sup_x\int_M |G(x,y)|\dV(y)<\infty.
\]
Using $|c_i-\rho_i|\le c_i+\rho_i$ and Tonelli's theorem, we obtain
\[
\iint_{M\times M}|G(x,y)|\,|c_i(x)-\rho_i|\,|c_j(y)-\rho_j|
\dV(x)\dV(y)<\infty.
\]
Finally, the Boltzmann entropy $H(c)$ is finite by assumption, with its negative part bounded by $s\log s\ge-e^{-1}$. The external-potential energy is finite because $\log(\rho_i h_i)$ is bounded.
\end{proof}

\begin{lemma}\label{lem:variation}
For a smooth positive density $c$ of prescribed masses and $\zeta\in C^\infty(M)^n$ with $\int_M\zeta_i\dV=0$ for every $i$, the directional derivative, with $A$, $\rho$ and $h$ fixed, satisfies
\[
D\F(c)[\zeta]
:=\left.\frac{\mathrm d}{\mathrm d\varepsilon}
\F(c+\varepsilon\zeta)\right|_{\varepsilon=0}
=\sum_i\int_M\mu_i(c)\zeta_i\dV,
\]
where
\begin{align}\label{def:mu}
\mu_i(c)=\log\frac{c_i}{\rho_i h_i}+1-U_i[c].
\end{align}
Along a smooth positive solution of \eqref{eq:flow}, the free-energy dissipation identity reads
\begin{equation}\label{eq:dissipation}
\frac{\mathrm d}{\mathrm dt}\F(c(t))=-\D(c(t)),\qquad
\D(c)=\sum_i\int c_i|\nabla_g\mu_i(c)|^2\dV.
\end{equation}
\end{lemma}
\begin{proof}
The zero-integral condition preserves each mass under $c\mapsto c+\varepsilon\zeta$. Since $c$ is smooth and positive on the compact surface, this perturbation remains positive for sufficiently small $|\varepsilon|$. Differentiating the Boltzmann entropy and the external-potential energy gives $\sum_i\int(\log(c_i/(\rho_i h_i))+1)\zeta_i\dV$. By symmetry of $A$ and $G$, the derivative of the interaction energy satisfies
\[
\begin{aligned}
D(-W_A)(c)[\zeta]
&=-\frac12\sum_{i,j}a_{ij}
\left(
\int\zeta_i\G(c_j-\rho_j)\dV
+\int(c_i-\rho_i)\G\zeta_j\dV
\right)\\
&=-\sum_i\int U_i[c]\zeta_i\dV.
\end{aligned}
\]
Combining these derivatives in \eqref{eq:energy} gives the variation formula. Integrating \eqref{eq:flow} on the closed surface gives
\[
\int_M\partial_t c_i\dV
=\frac{\mathrm d}{\mathrm dt}\int_M c_i\dV=0.
\]
We can therefore apply the first-variation formula with $\zeta_i=\partial_t c_i$.  Since $\mu_i$ differs from $\log(c_i/h_i)-U_i[c]$ by the constant $1-\log\rho_i$, integration by parts gives
\[
\frac{\mathrm d}{\mathrm dt}\F(c(t))
=\sum_i\int\mu_i\Div(c_i\nabla_g\mu_i)\dV
=-\sum_i\int c_i|\nabla_g\mu_i|^2\dV.\qedhere
\]
\end{proof}

\begin{remark}\label{rem:metric}
We explain the relation between \eqref{eq:metric} and
\eqref{eq:flow}.
At a smooth positive density $c$, every smooth variation $\zeta$ with $\int_M\zeta_i\dV=0$ can be written as
\[
\zeta_i=-\Div(c_i\nabla_g\psi_i).
\]
Indeed, the operator $-\Div(c_i\nabla_g\cdot)$ is uniformly elliptic, its kernel consists of the constants, and the zero-integral condition ensures solvability. Each $\psi_i$ is unique up to an additive constant.

Writing $p_i=c_i/\rho_i$, the quadratic form corresponding to \eqref{eq:metric} is \cite[Section~2.6]{AS}\cite[Section~1]{Otto}
\[
\sum_i\rho_i\int p_i|\nabla_g\psi_i|^2\dV
=\sum_i\int c_i|\nabla_g\psi_i|^2\dV.
\]
By Lemma~\ref{lem:variation} and integration by parts,
\[
D\F(c)[\zeta]
=\sum_i\int\mu_i(c)\zeta_i\dV
=\sum_i\int c_i\nabla_g\mu_i(c)\cdot\nabla_g\psi_i\dV.
\]
Thus, with respect to this quadratic form, the negative gradient is $\bigl(\Div(c_i\nabla_g\mu_i)\bigr)_i$, giving \eqref{eq:flow}.

Without the factors $\rho_i$ in \eqref{eq:metric}, the corresponding calculation uses
$\mathcal E(p)=\F(\rho_1p_1,\ldots,\rho_np_n)$. For a smooth variation $\eta$ with $\int_M\eta_i\dV=0$,
\[
D\mathcal E(p)[\eta]
=\sum_i\rho_i\int_M\mu_i(\rho p)\eta_i\dV.
\]
It would therefore give $\partial_t p_i=\Div(p_i\nabla_g(\rho_i\mu_i))$, introducing an additional factor $\rho_i$ in the $i$th density equation.
\end{remark}

\begin{lemma}\label{lem:stationary}
For a smooth positive density $c$ with $\int_M c_i\dV=\rho_i$, the following conditions are equivalent.
\begin{enumerate}
\item $c$ is a constrained critical point of $\F$.
\item $c$ is stationary for \eqref{eq:flow}.
\item For each $i\in I$, we get
\begin{equation}\label{eq:stationary}
c_i=\rho_i\frac{h_i e^{U_i[c]}}{\int_M h_i e^{U_i[c]}\dV}.
\end{equation}
\end{enumerate}
Under these conditions, $u_i=U_i[c]$ solves \eqref{eq:elliptic} and has zero mean.
\end{lemma}
\begin{proof}
Vanishing of the variation against all zero-mean test functions means that each $\mu_i$ is constant. Solving this relation for $c_i$ and imposing its mass gives \eqref{eq:stationary}. Conversely, that formula makes each $\mu_i$ constant, so both the constrained variation and the right-hand side of the flow vanish. If the flow is stationary, testing it with $\mu_i$ gives $\D(c)=0$. Positivity and connectedness imply that each $\mu_i$ is constant. Finally, substitution into \eqref{eq:poisson} gives \eqref{eq:elliptic}.
\end{proof}

\section{Local existence and continuation}\label{sec:local}
{\phantomsection}\label{not:time}
On a time interval $[a,b]$, $C_tC_x=C([a,b];C(M)^n)$ with $\|c\|_{C_tC_x}=\sup_{a\le t\le b}\max_i\|c_i(t)\|_\infty$. Parabolic H\"older spaces $C^{k+\alpha,(k+\alpha)/2}$ use the usual local-coordinate norms with distance $|x-y|+|t-s|^{1/2}$. One time derivative has the same order as two spatial derivatives.

\begin{lemma}\label{lem:heat}
Let $S(t)=e^{t\Delta_g}$ and let $K_t(x,y)$ be its heat kernel on $M$. For $0<t\le1$ and $1<q<\infty$, the following estimates hold for scalar functions $f$ and tangent vector fields $X$ with finite right-hand sides,
\begin{equation}\label{eq:heat}
\begin{split}
\norm{S(t)f}_\infty&\le C_qt^{-1/q}\norm f_q,\\
\norm{S(t)\Div X}_\infty&\le C_qt^{-1/2-1/q}\norm X_q,\\
\norm{S(t)\Div X}_\infty&\le Ct^{-1/2}\norm X_\infty.
\end{split}
\end{equation}
For nonnegative integers $j,k,\ell$, there are constants depending only on $(M,g)$ and these orders such that
\begin{equation}\label{eq:kernel}
|\partial_t^j\nabla_x^k\nabla_y^\ell K_t(x,y)|
\le C_{j,k,\ell}t^{-1-j-(k+\ell)/2}e^{-d_g(x,y)^2/(Ct)}.
\end{equation}
\end{lemma}
\begin{proof}
The kernel estimate \eqref{eq:kernel} follows from Ludewig \cite[Theorems 1.1 and 3.5]{Ludewig}. 

By compactness, we find a $C_V$ such that
\[
\operatorname{vol}_g(B_s(x))\le C_Vs^2
\qquad(x\in M,\ s>0).
\]
For $a>0$, set $A_k=\{y\in M:k\sqrt t\le d_g(x,y)<(k+1)\sqrt t\}$.
Then
\[
\begin{aligned}
\int_M e^{-a d_g(x,y)^2/t}\dV(y)
&=\sum_{k=0}^\infty\int_{A_k}e^{-a d_g(x,y)^2/t}\dV(y)\\
&\le\sum_{k=0}^\infty e^{-ak^2}
  \operatorname{vol}_g(B_{(k+1)\sqrt t}(x))\\
&\le C_Vt\sum_{k=0}^\infty(k+1)^2e^{-ak^2}
=:C_at,
\end{aligned}
\]
where $C_a<\infty$ is independent of $x$ and $t$. For $1<q<\infty$, let $q'=q/(q-1)$. The heat-kernel estimates \eqref{eq:kernel} and the preceding integral bound give, uniformly in $x\in M$,
\[
\begin{aligned}
\|K_t(x,\cdot)\|_{L^{q'}}
&\le Ct^{-1}
\left(\int_M e^{-q'd_g(x,y)^2/(Ct)}\dV(y)\right)^{1/q'}
\le C_qt^{-1+1/q'}=C_qt^{-1/q},\\
\|\nabla_yK_t(x,\cdot)\|_{L^{q'}}
&\le Ct^{-3/2}
\left(\int_M e^{-q'd_g(x,y)^2/(Ct)}\dV(y)\right)^{1/q'}
\le C_qt^{-3/2+1/q'}=C_qt^{-1/2-1/q}.
\end{aligned}
\] 
Integrating by parts, we get $S(t)\Div X(x)=-\int\langle\nabla_yK_t(x,y),X(y)\rangle_g\dV(y)$. H\"older's inequality proves \eqref{eq:heat}. The last estimate uses $q'=1$.
\end{proof}

We also use the zero-mean elliptic estimates
\begin{equation}\label{eq:ell-est}
\norm{\G f}_{W^{2,q}}\le C_q\norm f_q,\qquad
\norm{\G f}_{C^{k+2,\alpha}}\le C_{k,\alpha}\norm f_{C^{k,\alpha}}.
\end{equation}
These follow from the global $W^{2,q}$ and Schauder estimates for $-\Delta_g$ on the closed surface $M$. The zero-mean normalization of $\G f$ removes the constant kernel and allows the lower-order terms to be controlled by the right-hand side.

For the equation
$\partial_t c_i=\Delta_g c_i-\Div(c_i b_i[c])$, where
$b_i[c]=\nabla_g\log h_i+\nabla_g U_i[c]$, \emph{a continuous mild solution}
on $[a,a+L]$ means a vector-valued function $c\in C([a,a+L];C(M)^n)$ satisfying
\begin{equation}\label{eq:mild}
c_i(t)=S(t-a)c_i(a)
-\int_a^t S(t-r)\Div\bigl(c_i(r)b_i[c(r)]\bigr)\,\mathrm dr
\end{equation}
for every $t\in[a,a+L]$ and every $i\in I$. The integral is taken in $C(M)$, with $S(t-r)\Div$ interpreted by the heat kernel. The bound in \eqref{eq:heat} makes the time singularity integrable for continuous bounded $c$ of fixed masses. This definition does not require derivatives of $c$ to exist.

\begin{lemma}\label{lem:smoothing}
Let $c$ be a continuous mild solution of $\partial_t c_i=\Delta_g c_i-\Div(c_i b_i[c])$ on $[a,a+L]$, where $0<L\le1$, $b_i[c]=\nabla_g\log h_i+\nabla_g U_i[c]$ and the component masses are fixed. If $\sup_{[a,a+L]}\norm c_\infty\le B$, then for $0<\delta<L$, every integer $m\ge0$ and $0<\alpha<1$,
\[
\sup_{a+\delta\le t\le a+L}\sum_i
\left(\norm{c_i(t)}_{C^{m,\alpha}}+\norm{U_i[c(t)]}_{C^{m+2,\alpha}}\right)
\le C(B,L,\delta,m,\alpha).
\]

\end{lemma}
\begin{proof}
\par\smallskip
\noindent\emph{Step 1. H\"older continuity.}
Set $F_i=c_i b_i[c]$ and $q_i=\Delta_g\log h_i-\sum_j a_{ij}(c_j-\rho_j)$. The elliptic estimates bound $b$, $q$ and $F$ in terms of $B$ and the fixed data. We study the solution in the sense of (\ref{eq:mild}). Put $T(r)X=S(r)\Div X$ and write
\[
w(t)=\int_a^tT(t-s)F(s)\,\mathrm ds,
\qquad M_F=\sup_{a\le s\le a+L}\|F(s)\|_\infty.
\]
Integrating the heat-kernel derivative estimates \eqref{eq:kernel} in the $y$ variable gives, for $0<r\le1$,
\begin{align}\label{ineq:Txinfty}
\|T(r)X\|_\infty\le Cr^{-1/2}\|X\|_\infty.
\end{align}
Differentiating the kernel representation in $x$ and using the Gaussian integral bound, we obtain
\begin{align}\label{ineq:NablaTxinfty}
\|\nabla_gT(r)X\|_\infty
&\le \|X\|_\infty
\sup_{x\in M}\int_M|\nabla_x\nabla_yK_r(x,y)|_g\,\dV(y)\nonumber\\
&\le Cr^{-2}\|X\|_\infty
\sup_{x\in M}\int_M e^{-d_g(x,y)^2/(Cr)}\,\dV(y)\nonumber\\
&\le Cr^{-1}\|X\|_\infty.
\end{align}
A similar calculation using the estimate for $\partial_r\nabla_yK_r$ gives
\begin{align}\label{ineq:PartialTTxinfty}
\|(\partial_rT(r))X\|_\infty
\le Cr^{-3/2}\|X\|_\infty.
\end{align}
First we have the H\"older bound in spatial variables. For $d=d_g(x,y)<1$, (\ref{ineq:Txinfty}) and (\ref{ineq:NablaTxinfty}) yield
\[
|w(t,x)-w(t,y)|
\le CM_F\int_0^L\min\{r^{-1/2},dr^{-1}\}\,\mathrm dr
\le C_LM_Fd(1+|\log d|).
\]
Now we study the H\"older bound in the time variable. For $a\le t<t+h\le a+L$, (\ref{ineq:PartialTTxinfty}) yields
\[
\|w(t+h)-w(t)\|_\infty
\le CM_F\left(\int_0^h r^{-1/2}\,\mathrm dr
+h\int_h^L r^{-3/2}\,\mathrm dr\right)
\le CM_F\sqrt h.
\]
Since $d(1+|\log d|)\le C_\gamma d^\gamma$ for $d<1$ and the initial term $S(t-a)c(a)$ is smooth away from $t=a$, we obtain $c\in C^{\gamma,\gamma/2}$ on interior time intervals for every $0<\gamma<1$. The elliptic estimates \eqref{eq:ell-est} give $b$ the same time regularity and a uniform spatial $C^{1,\gamma}$ bound there. Hence $F=cb$ is spatially $C^\gamma$ uniformly in time.
\par\smallskip
\noindent\emph{Step 2. Classical solutions.}
Fix $b\in(a,a+L)$ and choose $a<a_0<b$. We establish the required regularity on $[b,a+L]$ by applying the mild formulation \eqref{eq:mild} with initial time $a_0$. Fix $0<\gamma<1$.
By Step~1,
\[
M_\gamma:=\sup_{a_0\le s\le a+L}\|F(s)\|_{C^\gamma}<\infty.
\]
Write
\[
c(t)=S(t-a_0)c(a_0)-w(t),
\qquad
w(t)=\int_{a_0}^tT(t-s)F(s)\,\mathrm ds.
\]
The initial semigroup term is smooth, with controlled derivatives, on every time interval whose left endpoint is greater than $a_0$. It therefore suffices first to estimate $w$.

We work component-wise and initially hold the time variable fixed. To use the spatial H\"older continuity of the vector field $F$, choose a finite coordinate cover and a subordinate smooth partition of unity $\{\chi_\ell\}$. In each chart write $F=\sum_aF_\ell^a\partial_{x_\ell^a}$, and choose a smooth cutoff $\psi_\ell$ supported in that chart and equal to one near $\operatorname{supp}\chi_\ell$. Set
\[
f_{\ell,a}=\chi_\ell F_\ell^a,
\qquad
Y_{\ell,a}=\psi_\ell\partial_{x_\ell^a},
\]
extending both by zero outside the chart. Relabeling the pairs
$(\ell,a)$ by $\nu$ gives
\[
F=\sum_\nu f_\nu Y_\nu,
\qquad
\sum_\nu\|f_\nu\|_{C^\gamma}\le C\|F\|_{C^\gamma},
\]
where the vector fields $Y_\nu$ are fixed and globally smooth.

Consider one term $fY$. Fix $x\in M$ and consider variable $z$. The kernel representation and integration by parts give
\[
\begin{aligned}
T(r)(fY)(z)
&=T(r)((f-f(x))Y)(z)+f(x)S(r)(\Div Y)(z)\\
&=-\int_M
\langle\nabla_yK_r(z,y),Y(y)\rangle_g
\bigl(f(y)-f(x)\bigr)\,\dV(y)
+f(x)S(r)(\Div Y)(z).
\end{aligned}
\]
We differentiate in $z$ while keeping $x$ fixed and only afterwards set $z=x$. Thus no derivatives of $f$ are required.

By \eqref{eq:kernel},
\[
|\nabla_z^k\partial_r^j\nabla_yK_r(z,y)|_{z=x}
\le Cr^{-1-j-(k+1)/2}
e^{-d_g(x,y)^2/(Cr)}.
\]
Moreover,
\[
|f(y)-f(x)|\le[f]_{C^\gamma}d_g(x,y)^\gamma.
\]
The elementary inequality
\[
d^\gamma e^{-d^2/(Cr)}
\le C_\gamma r^{\gamma/2}e^{-d^2/(2Cr)}
\qquad(d\ge0)
\]
and the Gaussian integral estimate proved above imply
\[
\begin{aligned}
\int_Md_g(x,y)^\gamma
e^{-d_g(x,y)^2/(Cr)}\,\dV(y)
\le C_\gamma r^{\gamma/2}
\int_Me^{-d_g(x,y)^2/(2Cr)}\,\dV(y)\le Cr^{1+\gamma/2}.
\end{aligned}
\]
Consequently,
\begin{align}
\Bigg|\nabla_z^k\partial_r^j
\int_M &
\langle\nabla_yK_r(z,y),Y(y)\rangle_g
\bigl(f(y)-f(x)\bigr)\,\dV(y)
\Bigg|_{z=x}
\Bigg|_g\nonumber\\
&\quad\le [f]_{C^\gamma}\|Y\|_\infty
\int_M
|\nabla_x^k\partial_r^j\nabla_yK_r(x,y)|_g
d_g(x,y)^\gamma\,\dV(y)\nonumber\\
&\quad\le C[f]_{C^\gamma}r^{-1-j-(k+1)/2}
\int_Md_g(x,y)^\gamma
e^{-d_g(x,y)^2/(Cr)}\,\dV(y)\nonumber\\
&\quad\le C[f]_{C^\gamma}
r^{-1-j-(k+1)/2}r^{1+\gamma/2}.\nonumber
\end{align}
For the remaining term, $\Div Y$ is fixed and smooth, and
\[
\partial_r^jS(r)(\Div Y)
=S(r)\Delta_g^j(\Div Y).
\]
Using $H^{k+2}(M)\hookrightarrow C^k(M)$ and the contractivity of $S(r)$ in the spectral Sobolev norm, we obtain
\[
\begin{aligned}
|\nabla_z^k\partial_r^j
\bigl[f(x)S(r)(\Div Y)(z)\bigr]|_{z=x}|_g&\le \|f\|_\infty
\|\nabla_g^kS(r)\Delta_g^j(\Div Y)\|_\infty\\
&\quad\le C\|f\|_\infty
\|S(r)\Delta_g^j(\Div Y)\|_{H^{k+2}}\\
&\quad\le C\|f\|_\infty
\|\Delta_g^j(\Div Y)\|_{H^{k+2}}
\le C_{k,j,Y}\|f\|_\infty,
\end{aligned}
\]
uniformly for $0<r\le1$, since $Y$ is fixed and smooth. Summing over the finite decomposition, we obtain
\begin{equation}\label{eq:holder-kernel}
\begin{aligned}
\|\nabla^k\partial_r^jT(r)F\|_\infty
&\le C\left(
r^{-1-j-(k+1)/2}r^{1+\gamma/2}+1
\right)\|F\|_{C^\gamma}\\
&\le Cr^{-(1+k+2j-\gamma)/2}\|F\|_{C^\gamma}
\end{aligned}
\end{equation}
for $k=0,1,2$, $j=0,1$ and $0<r\le1$. The last inequality uses $1+k+2j-\gamma>0$.

Fix $0<\beta<\gamma$. The interpolation estimate
\[
\|u\|_{C^{1,\beta}}
\le C\left(
\|u\|_{C^1}
+\|\nabla_gu\|_\infty^{1-\beta}
 \|\nabla_g^2u\|_\infty^\beta
\right)
\]
follows by comparing the gradient at two points by parallel transport and using the smaller of its supremum bound and its Lipschitz bound. See also \cite[Section~6.8]{GT}. Applying it to $u=T(r)F$ and using \eqref{eq:holder-kernel}, we get
\[
\begin{aligned}
\|T(r)F\|_{C^{1,\beta}}
&\le C\left[
r^{-1+\gamma/2}
+\bigl(r^{-1+\gamma/2}\bigr)^{1-\beta}
 \bigl(r^{-3/2+\gamma/2}\bigr)^\beta
\right]\|F\|_{C^\gamma}\\
&\le Cr^{-1+(\gamma-\beta)/2}\|F\|_{C^\gamma},
\end{aligned}
\]
since
\[
\left(-1+\frac{\gamma}{2}\right)(1-\beta)
+\left(-\frac32+\frac{\gamma}{2}\right)\beta
=-1+\frac{\gamma-\beta}{2}.
\]
The exponent is greater than $-1$. Hence
\[
\begin{aligned}
\|w(t)\|_{C^{1,\beta}}
\le CM_\gamma\int_0^{t-a_0}
r^{-1+(\gamma-\beta)/2}\,\mathrm dr=\frac{2CM_\gamma}{\gamma-\beta}
(t-a_0)^{(\gamma-\beta)/2}
\le C_{\gamma,\beta}M_\gamma.
\end{aligned}
\]
These estimates can first be applied to the integrals truncated at $s=t-\varepsilon$. The same integrable bound shows that the truncated integrals converge in $C^{1,\beta}$ as $\varepsilon\downarrow0$, and justifies differentiation under the integral sign once in space.

We next establish the time H\"older estimates. For $a_0\le t<t+h\le a+L$ and $k=0,1$, changing variables gives
\[
\begin{aligned}
\nabla^kw(t+h)-\nabla^kw(t)
=\int_0^h\nabla^kT(r)F(t+h-r)\,\mathrm dr+\int_0^{t-a_0}
\nabla^k[T(r+h)-T(r)]F(t-r)\,\mathrm dr.
\end{aligned}
\]
Set
\[
p_k=\frac{1+k-\gamma}{2}\in(0,1).
\]
Formula \eqref{eq:holder-kernel} gives
\[
\|\nabla^kT(r)F(s)\|_\infty
\le CM_\gamma r^{-p_k},
\qquad
\|\nabla^k\partial_rT(r)F(s)\|_\infty
\le CM_\gamma r^{-p_k-1}.
\]
The new interval contributes at most $CM_\gamma\int_0^h r^{-p_k}\,\mathrm dr$. On the old interval with $r<h$, estimate the two operators separately, using $(r+h)^{-p_k}\le r^{-p_k}$. On the part $r\ge h$, use
\[
[T(r+h)-T(r)]X
=\int_r^{r+h}(\partial_\sigma T(\sigma))X\,\mathrm d\sigma
\]
to obtain
\[
\begin{aligned}
\|\nabla^k[T(r+h)-T(r)]F(t-r)\|_\infty
\le CM_\gamma\int_r^{r+h}\sigma^{-p_k-1}\,\mathrm d\sigma\le CM_\gamma h r^{-p_k-1}.
\end{aligned}
\]
No time derivative of $F$ is used. Since $L\le1$, these bounds yield
\[
\begin{aligned}
\|\nabla^kw(t+h)-\nabla^kw(t)\|_\infty
&\le CM_\gamma\left(
3\int_0^h r^{-p_k}\,\mathrm dr
+h\int_h^1r^{-p_k-1}\,\mathrm dr
\right)\\
&\le CM_\gamma
\left(\frac{3}{1-p_k}+\frac{1}{p_k}\right)h^{1-p_k}.
\end{aligned}
\]
Taking $k=0$ and $k=1$, respectively, gives
\[
\begin{split}
\|w(t+h)-w(t)\|_\infty
&\le C_\gamma M_\gamma h^{(1+\gamma)/2},\\
\|\nabla_gw(t+h)-\nabla_gw(t)\|_\infty
&\le C_\gamma M_\gamma h^{\gamma/2}.
\end{split}
\]
Together with the spatial $C^{1,\beta}$ bound and the smoothness of $S(t-a_0)c(a_0)$ away from $a_0$, these estimates imply $c\in C^{1+\beta,(1+\beta)/2}$ on smaller time intervals. Indeed, $\gamma>\beta$ supplies the required time exponents $(1+\beta)/2$ for $c$ and $\beta/2$ for $\nabla_gc$.

We now obtain the second spatial derivatives and the first time derivative. Since
\[
\Div b_i
=\Delta_g\log h_i+\Delta_gU_i[c]
=\Delta_g\log h_i-\sum_j a_{ij}(c_j-\rho_j),
\]
the nondivergence form is 
\begin{equation}\label{eq:linear-form}
\partial_tc_i=\Delta_gc_i-b_i\cdot\nabla_gc_i-q_ic_i,
\qquad
q_i=\Delta_g\log h_i-\sum_j a_{ij}(c_j-\rho_j).
\end{equation}
At this stage the equation holds in distributions, as follows from \eqref{eq:mild}. The preceding spatial regularity justifies the identity $\Div(c_ib_i)=b_i\cdot\nabla_gc_i+q_ic_i$.

Choose $a_1$ with $a_0<a_1<b$ and work on $J=[a_1,a+L]$. Elliptic regularity gives uniform spatial H\"older bounds for $b_i$ and $q_i$. It also gives, for $s,t\in J$,
\[
\|b_i(t)-b_i(s)\|_\infty
+\|q_i(t)-q_i(s)\|_\infty
\le C\|c(t)-c(s)\|_\infty.
\]
For the first term, this follows from the $W^{2,q}$ estimate with any fixed $q>2$ and the two-dimensional Sobolev embedding. Consequently,
\[
b_i,\ q_i,\ c_i,\ \nabla_gc_i\in C^{\beta,\beta/2}(M\times J).
\]
Define
\[
H_i:=-b_i\cdot\nabla_gc_i-q_ic_i.
\]
The product estimate in the parabolic H\"older space yields
\[
\begin{aligned}
\|H_i\|_{C^{\beta,\beta/2}}
\le C\bigl(
&\|b_i\|_{C^{\beta,\beta/2}}
 \|\nabla_gc_i\|_{C^{\beta,\beta/2}}\\
&+\|q_i\|_{C^{\beta,\beta/2}}
 \|c_i\|_{C^{\beta,\beta/2}}
\bigr),
\end{aligned}
\]
so $(\partial_t-\Delta_g)c_i=H_i$ has a controlled $C^{\beta,\beta/2}$ right-hand side.

We record the interior Schauder estimate in coordinates. Write
\[
\Delta_gu=a^{ab}\partial_{ab}u+d_g^a\partial_au,
\qquad
a^{ab}=g^{ab},\quad
d_g^b=|g|^{-1/2}\partial_a(|g|^{1/2}g^{ab}),
\]
using the Einstein summation convention. For a classical solution of
\[
u_t-a^{ab}u_{ab}-d^au_a+qu=f
\]
and nested coordinate cylinders $Q'\Subset Q$, the interior Schauder estimate \cite[Theorem~4.9]{Lieberman} gives
\[
\|u\|_{C^{2+\beta,1+\beta/2}(Q')}
\le C\left(
\|u\|_{C^0(Q)}+\|f\|_{C^{\beta,\beta/2}(Q)}
\right).
\]
The constant depends on the ellipticity, the coefficient $C^{\beta,\beta/2}$ bounds and the separation of the cylinders.

To apply this estimate without assuming classical regularity in advance, we approximate the right-hand side of the fixed heat equation. Extend $H_i$ constantly in time beyond the endpoints of $J$ and mollify in a finite atlas. This gives smooth functions $H_{i,\varepsilon}$ satisfying
\[
H_{i,\varepsilon}\longrightarrow H_i
\quad\text{uniformly on }M\times J,
\qquad
\sup_\varepsilon
\|H_{i,\varepsilon}\|_{C^{\beta,\beta/2}(M\times J)}
\le C\|H_i\|_{C^{\beta,\beta/2}(M\times J)}.
\]
Define
\[
u_{i,\varepsilon}(t)
=S(t-a_1)c_i(a_1)
+\int_{a_1}^tS(t-s)H_{i,\varepsilon}(s)\,\mathrm ds.
\]
These functions are smooth for $t>a_1$ and solve
\[
(\partial_t-\Delta_g)u_{i,\varepsilon}
=H_{i,\varepsilon}.
\]

Testing the distributional equation for $c_i$ against a Laplace eigenfunction $e_\ell$ gives
\[
\frac{\mathrm d}{\mathrm dt}\langle c_i(t),e_\ell\rangle
+\lambda_\ell\langle c_i(t),e_\ell\rangle
=\langle H_i(t),e_\ell\rangle.
\]
Solving these scalar equations and using continuity at $a_1$ identifies
\[
c_i(t)
=S(t-a_1)c_i(a_1)
+\int_{a_1}^tS(t-s)H_i(s)\,\mathrm ds.
\]
Since $S(t)$ is contractive on $C(M)$,
\[
\sup_{t\in J}\|u_{i,\varepsilon}(t)-c_i(t)\|_\infty
\le L\|H_{i,\varepsilon}-H_i\|_{C^0(M\times J)}
\longrightarrow0.
\]
The interior Schauder estimate for the fixed heat operator, applied to $u_{i,\varepsilon}$, gives on smaller cylinders
\[
\sup_\varepsilon
\|u_{i,\varepsilon}\|_{C^{2+\beta,1+\beta/2}(Q')}
\le C.
\]
Compactness at any smaller H\"older exponent identifies the limit with $c_i$. Lower semicontinuity of the $\beta$-H\"older seminorms preserves the displayed bound. Thus \(c_i\in C^{2+\beta,1+\beta/2}\), and \eqref{eq:linear-form} holds classically.

For \eqref{eq:linear-form}, the coordinate coefficients are
\[
d^a=d_g^a-b_i^a,\qquad q=q_i.
\]
Applying the same estimate with $u=c_i$ and $f=0$ on a finite atlas gives the interior $C^{2+\beta,1+\beta/2}$ bound. We use backward parabolic cylinders ending at times at most $a+L$. The estimates at the upper endpoint follow by limits from earlier times. All bounds on $[b,a+L]$ depend only on the previously controlled norms, the fixed data and the chosen time separations.

\par\smallskip
\noindent\emph{Step 3. Higher derivatives.}
We argue by induction. On a larger time interval $J$, assume
\[
\|c\|_{C^{k+1+\gamma,(k+1+\gamma)/2}(M\times J)}
\le C_k,
\qquad k\ge1,\quad 0<\beta<\gamma<1.
\]
Set
\[
L_i=\partial_t-a^{ab}\partial_{ab}-d^a\partial_a+q_i.
\]
For $|\eta|=k$, formal differentiation of $L_ic_i=0$ gives
\[
L_i(D^\eta c_i)=R_{i,\eta},
\]
where
\begin{align}
R_{i,\eta}=
\sum_{0<\nu\le\eta}\binom{\eta}{\nu}
(D^\nu a^{ab})D^{\eta-\nu}\partial_{ab}c_i+\sum_{0<\nu\le\eta}\binom{\eta}{\nu}
(D^\nu d^a)D^{\eta-\nu}\partial_ac_i
-\sum_{0<\nu\le\eta}\binom{\eta}{\nu}
(D^\nu q_i)D^{\eta-\nu}c_i.\nonumber
\end{align}
We set $R_{i,0}=0$. The induction hypothesis and elliptic regularity give
\[
\begin{aligned}
\sup_{t\in J}\|\partial_tc(t)\|_{C^{k-1,\gamma}}
&\le C_k,\\
\|c(t)-c(s)\|_{C^{k-1,\beta}}
&\le C_k|t-s|,\\
\|U(t)-U(s)\|_{C^{k+1,\beta}}
&\le C\|c(t)-c(s)\|_{C^{k-1,\beta}}
\le C_k|t-s|.
\end{aligned}
\]
Moreover,
\[
\|D^kc(t)-D^kc(s)\|_\infty
\le C_k|t-s|^{(1+\gamma)/2}.
\]
Together, they give
\[
\|R_{i,\eta}\|_{C^{\beta,\beta/2}}\le C_k.
\]

To study the differentiation, write $\eta=\xi+e_\ell$ and set
\[
v=D^\xi c_i,\qquad
v_h=\delta_hv
:=\frac{v(t,x+he_\ell)-v(t,x)}{h}.
\]
Here $|\xi|=k-1$, so $v\in C^{2+\gamma,1+\gamma/2}$ by the induction hypothesis. We obtain
\[
(v_h)_t-a_h^{ab}(v_h)_{ab}-d_h^a(v_h)_a+q_{i,h}v_h
=G_h,
\]
where the coefficients with subscript $h$ are evaluated at $(t,x+he_\ell)$, and
\[
G_h=\delta_hR_{i,\xi}
+(\delta_ha^{ab})v_{ab}
+(\delta_hd^a)v_a-(\delta_hq_i)v.
\]
On smaller coordinate cylinders, the preceding bounds give
\[
\sup_{0<|h|<h_0}
\left(
\|v_h\|_{C^0(Q)}
+\|G_h\|_{C^{\beta,\beta/2}(Q)}
\right)\le C_k.
\]
Indeed, $\partial_\ell R_{i,\xi}$ involves at most $k+1$ spatial derivatives of $c$. The interior Schauder estimate therefore yields
\[
\sup_{0<|h|<h_0}
\|v_h\|_{C^{2+\beta,1+\beta/2}(Q')}
\le C_k,
\qquad Q'\Subset Q.
\]
Passing to the limit $h\to0$, using compactness and the lower semicontinuity of the H\"older seminorms, we get
\[
D^\eta c_i\in C^{2+\beta,1+\beta/2}(Q'),
\qquad
L_i(D^\eta c_i)=R_{i,\eta}.
\]

The mixed time derivatives are controlled by
\begin{align}
\partial_t^{\ell+1}c_i
=\Delta_g\partial_t^\ell c_i-\sum_{r=0}^{\ell}\binom{\ell}{r}
\left[
(\partial_t^rb_i)\cdot\nabla_g\partial_t^{\ell-r}c_i
+(\partial_t^rq_i)\partial_t^{\ell-r}c_i
\right],\nonumber
\end{align}
together with, for $\ell\ge1$,
\[
-\Delta_g\partial_t^\ell U_i
=\sum_j a_{ij}\partial_t^\ell c_j,
\qquad
\int_M\partial_t^\ell U_i\,\mathrm dV_g=0.
\]
Thus elliptic regularity and the preceding spatial estimates give, on a smaller time interval $J'$,
\[
\|c\|_{C^{k+2+\beta,(k+2+\beta)/2}(M\times J')}
\le C_{k+1},
\]
which advances the induction by one spatial order. For any prescribed $m$ and $\alpha\in(0,1)$, the above arguments give
\[
\sup_{a+\delta\le t\le a+L}\sum_i
\left(
\|c_i(t)\|_{C^{m,\alpha}}
+\|U_i[c(t)]\|_{C^{m+2,\alpha}}
\right)
\le C(B,L,\delta,m,\alpha).
\]
The constant also depends on the fixed equation and geometric data.

\end{proof}

\begin{remark}\label{rem:smoothing}
The proof of Lemma~\ref{lem:smoothing} also gives all finite parabolic derivative bounds on smaller time intervals. The constants may depend on the fixed geometric and equation data, but they are independent of the initial time $a$ and of a positive lower bound for $c$.
\end{remark}

\begin{lemma}\label{lem:local}
For smooth positive initial densities of the prescribed masses, there exists a unique maximal smooth positive solution on $[0,T_{\max})$. If $T_{\max}<\infty$, then
\[
\lim_{t\uparrow T_{\max}}\sum_i\norm{c_i(t)}_\infty=\infty.
\]
\end{lemma}
\begin{proof}
\par\smallskip
\noindent\emph{Step 1. Construction of a mild solution in short time.} We apply Banach fixed point theorem to achieve this goal. Set
\[
b_i[c]=\nabla_g\log h_i+\nabla_gU_i[c],\qquad F_i(c)=c_i b_i[c].
\] 
By the Sobolev embedding of dimension two and the standard elliptic $W^{2,q}$ estimates on the closed surface $M$ \cite[Chapter~9]{GT}, we get for fixed $q>2$

\begin{align}
\|b[c]-b[d]\|_\infty
\le C\|U[c]-U[d]\|_{W^{2,q}}\le C\|c-d\|_q\le C\|c-d\|_\infty.\nonumber
\end{align}

Here $c$ and $d$ have the same prescribed masses. The product identity $F_i(c)-F_i(d)=(c_i-d_i)b_i[c]+d_i(b_i[c]-b_i[d])$ therefore yields, on any fixed bounded set of densities,

\[
\norm{b[c]-b[d]}_\infty\le C\norm{c-d}_\infty,\qquad
\norm{F(c)-F(d)}_\infty\le L_B\norm{c-d}_\infty.
\]
Take $B=\norm{c_0}_\infty+1$. In $C([0,T];C(M)^n)$ consider the closed set
\[
\mathcal B={\Bigg\{}c:c(0)=c_0,\ \int c_i(t)\dV=\rho_i,\
\sup_{0\le t\le T}\norm{c(t)}_\infty\le B{\Bigg\}}.
\]
It is nonempty because it contains $S(t)c_0$. It is complete in the supremum distance. Define
\[
\Phi(c)_i(t)=S(t)c_{i,0}-\int_0^t S(t-s)\Div F_i(c(s))\,\mathrm ds.
\]

Recall the time-space norm introduced at the beginning of Section~\ref{sec:local} (p.~\pageref{not:time}). On $[0,T]$, it is
\[
\|v\|_{C_tC_x}=\sup_{0\le t\le T}\max_{i\in I}\|v_i(t)\|_\infty.
\]
Let $M_B$ bound $\norm{F(c)}_\infty$ on the ball. By \eqref{eq:heat},
\[\norm{\Phi(c)-S(\cdot)c_0}_{C_tC_x}\le2CM_B\sqrt T\text{ and }\norm{\Phi(c)-\Phi(d)}_{C_tC_x}\le2CL_B\sqrt T\norm{c-d}_{C_tC_x}.\]
Choose $T\le1$ so that $2CM_B\sqrt T\le1$ and $2CL_B\sqrt T\le\tfrac12$. By self-adjointness and $S(r)1=1$,
\[
\int_M S(r)\Div X\,\dV
=-\int_M\langle X,\nabla_gS(r)1\rangle_g\,\dV=0.
\]
Consequently, $\int_M\Phi(c)_i(t)\,\dV=\rho_i$ and $\|\Phi(c)(t)\|_\infty \le\|c_0\|_\infty+2CM_B\sqrt T\le B$. This implies that $\Phi:\mathcal{B}\to\mathcal{B}$. To verify the contraction, we write
\[
W_c(t):=\int_0^tS(t-s)\Div F(c(s))\,\mathrm ds.
\]
The estimate
\[
\|W_c(t)\|_\infty
\le CM_B\int_0^t(t-s)^{-1/2}\,\mathrm ds
=2CM_B\sqrt t
\]
and strong continuity of $S(t)$ on $C(M)$ give
\[
\|\Phi(c)(t)-c_0\|_\infty
\le\|S(t)c_0-c_0\|_\infty+2CM_B\sqrt t
\longrightarrow0\mbox{ as }t\to0.
\]
For $\varepsilon>0$, set
\[
W_{c,\varepsilon}(t)
:=\int_0^{\max\{t-\varepsilon,0\}}
S(t-s)\Div F(c(s))\,\mathrm ds.
\]
Since the semigroup time stays away from zero, $W_{c,\varepsilon}\in C([0,T];C(M)^n)$ by heat-kernel smoothing and strong semigroup continuity. Moreover,
\begin{align}
\sup_{0\le t\le T}
\|W_c(t)-W_{c,\varepsilon}(t)\|_\infty
\le CM_B\sup_{0\le t\le T}
\int_{\max\{t-\varepsilon,0\}}^t(t-s)^{-1/2}\,\mathrm ds\le2CM_B\sqrt\varepsilon.\nonumber
\end{align}
Thus $W_c$, and hence $\Phi(c)$, belongs to $C([0,T];C(M)^n)$. This proves $\Phi(\mathcal B)\subseteq\mathcal B$. Furthermore,
\[
\begin{aligned}
\|\Phi(c)-\Phi(d)\|_{C_tC_x}
&\le CL_B\sup_{0\le t\le T}
\int_0^t(t-s)^{-1/2}\|c(s)-d(s)\|_\infty\,\mathrm ds\\
&\le2CL_B\sqrt T\,\|c-d\|_{C_tC_x}
\le\tfrac12\|c-d\|_{C_tC_x}.
\end{aligned}
\]
Therefore, $\Phi$ is a contraction. Banach's fixed point theorem therefore gives a unique fixed point in $\mathcal B$. The unique existence result  follows.

\par\smallskip
\noindent\emph{Step 2. Smoothness and positivity.}
Smoothness at positive times follows from Lemma~\ref{lem:smoothing} on each compact subinterval.

For regularity up to $t=0$, fix an integer $s>2$. We construct the same solution in $C([0,T_s];H^s(M)^n)$ for a sufficiently small $T_s>0$. Elliptic regularity \cite[Section~6.3.1, Theorem~2]{Evans} gives
\[
\|\nabla_gU[c]\|_{H^{s+1}}\le C_s\|c\|_{H^s},
\qquad
\|\nabla_g(U[c]-U[d])\|_{H^{s+1}}
\le C_s\|c-d\|_{H^s}.
\]
Since $H^s(M)$ is an algebra, we get the identity
\[
F_i(c)-F_i(d)
=(c_i-d_i)b_i[c]+d_i\bigl(b_i[c]-b_i[d]\bigr)
\]
and the inequalities
\[
\begin{aligned}
\|F(c)\|_{H^s}
&\le C_s\|c\|_{H^s}(1+\|c\|_{H^s}),\\
\|F(c)-F(d)\|_{H^s}
&\le C_s\bigl(1+\|c\|_{H^s}+\|d\|_{H^s}\bigr)
\|c-d\|_{H^s}.
\end{aligned}
\]
Here vector-field Sobolev norms are defined in a fixed finite atlas.

We next establish the semigroup estimate needed for the contraction. Write $f_\ell=\langle f,e_\ell\rangle$, where $-\Delta_ge_\ell=\lambda_\ell e_\ell$ is the orthonormal eigenbasis used to define the Sobolev norms. For $0<t\le1$,
\begin{align}
\|S(t)f\|_{H^s}^2=\sum_\ell(1+\lambda_\ell)^s
e^{-2t\lambda_\ell}|f_\ell|^2\le
\sup_{\lambda\ge0}\bigl((1+\lambda)e^{-2t\lambda}\bigr)
\|f\|_{H^{s-1}}^2
\le Ct^{-1}\|f\|_{H^{s-1}}^2,\nonumber
\end{align}
because
\(
(1+\lambda)e^{-2t\lambda}
\le1+(2et)^{-1}\le Ct^{-1}\).
Moreover, the coordinate formula $\Div X=|g|^{-1/2}\partial_a(|g|^{1/2}X^a)$ gives
\[
\|\Div X\|_{H^{s-1}}\le C_s\|X\|_{H^s(TM)}.
\]
Taking square roots and combining these estimates yields
\[
\|S(t)\Div X\|_{H^s}
\le C_st^{-1/2}\|X\|_{H^s(TM)}.
\]
The estimate extends from smooth vector fields to $H^s(TM)$ by density.

Set
\[
B_s=\|c_0\|_{H^s}+1,\qquad
\mathcal X_s=C([0,T_s];H^s(M)^n),
\qquad
\|c\|_{\mathcal X_s}
=\sup_{0\le t\le T_s}\|c(t)\|_{H^s},
\]
and consider the closed subset
\[
\mathcal B_s=
\left\{
c\in\mathcal X_s:
c(0)=c_0,\quad
\int_Mc_i(t)\,\dV=\rho_i,\quad
\|c\|_{\mathcal X_s}\le B_s
\right\}.
\]
It contains $S(\cdot)c_0$. Let $M_s$ and $L_s$ be the bound and Lipschitz constant for $F$ on the $H^s$ ball of radius $B_s$. For the same map $\Phi$ as in Step~1, the preceding estimates give
\[
\begin{aligned}
\|\Phi(c)-S(\cdot)c_0\|_{\mathcal X_s}
&\le C_sM_s\sup_{0\le t\le T_s}
\int_0^t(t-r)^{-1/2}\,\mathrm dr
\le2C_sM_s\sqrt{T_s},\\
\|\Phi(c)-\Phi(d)\|_{\mathcal X_s}
&\le2C_sL_s\sqrt{T_s}\,
\|c-d\|_{\mathcal X_s}.
\end{aligned}
\]
Arguing as in Step~1 in $\mathcal X_s$, we obtain a fixed point of $\Phi$ in $\mathcal B_s$ for sufficiently small $T_s>0$. By uniqueness, it agrees with the previously constructed solution, so $c\in C([0,T_s];H^s(M)^n)$.

Put $N(c)=\Delta_gc-\Div F(c)$. Since $\partial_tc=N(c)$ in distributions and $N(c)\in C([0,T_s];H^{s-2})$, we have
\[
c(t)=c_0+\int_0^tN(c(r))\,\mathrm dr
\quad\text{in }H^{s-2},
\]
and hence $c\in C^1([0,T_s];H^{s-2})$. The product and elliptic estimates make $N:H^r\to H^{r-2}$ smooth on the space of functions with the prescribed masses. Iterating at lower Sobolev orders gives $c\in C^j([0,T_s];H^{s-2j})$ for $s-2j>2$. For any fixed $m,j$, choosing $s-2j>\max\{m+1,2\}$ and applying Sobolev embedding yields $\partial_t^jc\in C([0,T_s];C^m)$.
Together with Lemma~\ref{lem:smoothing} on positive-time intervals, this proves smoothness up to $t=0$.

We prove positivity by the weak parabolic maximum principle \cite[Section~7.1.4, Theorem~9]{Evans}, giving the comparison argument directly. Fix $T$ within the existence interval and choose $Q\ge\max\{0,\sup_{i,[0,T]\times M}q_i\}$. Set $v_i(t)=(\min_M c_{i,0})e^{-Qt}$ and $P_i=\partial_t-\Delta_g+b_i\cdot\nabla_g+q_i$. Then $P_ic_i=0$ and $P_iv_i=(q_i-Q)v_i\le0$, so $w_i=c_i-v_i$ satisfies $P_iw_i\ge0$ and $w_i(0)\ge0$. Choose $K=1+\sup_{i,[0,T]\times M}|q_i|$ and put $z_i=e^{-Kt}w_i$. It follows that
\[
\partial_tz_i-\Delta_gz_i+b_i\cdot\nabla_gz_i+(q_i+K)z_i\ge0,
\qquad q_i+K\ge1.
\]
If $z_i$ had a negative minimum on $[0,T]\times M$, it would be attained at $(t_*,x_*)$ with $t_*>0$. At that point, $\partial_tz_i\le0$, $\nabla_gz_i=0$ and $\Delta_gz_i\ge0$, so the left-hand side would be strictly negative, a contradiction. There is no spatial boundary because $M$ is closed. Hence $z_i\ge0$ and $c_i\ge v_i>0$ throughout $[0,T]\times M$.

\par\smallskip
\noindent\emph{Step 3. Continuation.}
Finally, the construction restarts with a time length depending only on a supremum bound and the fixed data. If a sequence of times tending to a finite $T_{\max}$ had bounded suprema, restarting at those times would extend the unique solution beyond $T_{\max}$. This proves the stated continuation criterion.
\end{proof}

\section{Uniform estimates and global existence}\label{sec:global}
Throughout this section, the coupling matrix satisfies \eqref{hyp:A} and the masses lie in the strictly subcritical region 
\[
\Lambda_J(\rho)>0\qquad(\varnothing\ne J\subseteq I).
\]
We prove that the maximal solution from Lemma~\ref{lem:local} is global in time, has uniform spatial regularity and stays uniformly positive away from the initial time. The proof proceeds from entropy coercivity to uniform $L^p$ and supremum bounds. We first state the logarithmic HLS estimate for a general matrix $B$, since it will also be applied to a rescaling of $A$.

\begin{lemma}\label{lem:HLS}
Let $B=(b_{ij})$ be a symmetric matrix with nonnegative entries and let $\rho_i>0$. Define $W_B$ by \eqref{eq:energy} with $A$ replaced by $B$. Suppose that
\[
\Lambda_J^B(\rho):=8\pi\sum_{i\in J}\rho_i
-\sum_{i,j\in J}b_{ij}\rho_i\rho_j>0
\qquad(\varnothing\ne J\subseteq I).
\]
There exists a constant $C_B=C(M,g,B,\rho)>0$, independent of $c$, such that
\[
H(c)-W_B(c)\ge-C_B
\]
for every $c=(c_i)_{i\in I}$ satisfying
\[
c_i\ge0\ \text{a.e. in }M,\qquad
\int_Mc_i\dV=\rho_i,\qquad
\int_Mc_i|\log c_i|\dV<\infty
\quad(i\in I).
\]
\end{lemma}
\begin{proof}

We apply Shafrir and Wolansky \cite[Proposition 2.1]{SW}. In their notation, the dimension is $m$, the masses are $M_i$ and the interaction coefficients are $a_{ij}$. Their strict subsystem condition is
\[
m\sum_{i\in J}M_i-\sum_{i,j\in J}a_{ij}M_iM_j>0
\qquad(\varnothing\ne J\subseteq I).
\]
Here $J$ ranges over all nonempty subsets, including the full index set. Taking $m=2$, $M_i=\rho_i$ and $a_{ij}=b_{ij}/(4\pi)$ gives
\[
2\sum_{i\in J}\rho_i
-\frac1{4\pi}\sum_{i,j\in J}b_{ij}\rho_i\rho_j
=\frac{\Lambda_J^B(\rho)}{4\pi}>0.
\]
Thus their hypothesis holds for every subsystem. The inequality gives
\[
H(c)+\frac1{4\pi}\sum_{i,j}b_{ij}
\iint c_i(x)\log d_g(x,y)c_j(y)\dV(x)\dV(y)\ge-C.
\]
By symmetry and the zero-mean normalization of $G$, $\int(c_i-\rho_i)\G(c_j-\rho_j)\dV=\int c_i\G c_j\dV$. Formula \eqref{eq:singularity} changes the preceding expression by at most $\frac12\norm R_\infty\sum_{i,j}b_{ij}\rho_i\rho_j$. This proves the claim.
\end{proof}

\begin{proposition}\label{prop:coercivity}

Assume \eqref{hyp:A} and $\Lambda_J(\rho)>0$ for every nonempty $J\subseteq I$. There exist $\varepsilon=\varepsilon(A,\rho)>0$ and $C=C(M,g,A,\rho,h)>0$, independent of $c$, such that every density satisfying the mass and entropy conditions in Lemma~\ref{lem:HLS} obeys

\begin{equation}\label{eq:coercivity}
\F(c)\ge\varepsilon H(c)-C.
\end{equation}
\end{proposition}
\begin{proof}
For each nonempty $J$, write $S_J=\sum_{i\in J}\rho_i$ and $Q_J=\sum_{i,j\in J}a_{ij}\rho_i\rho_j$. Since the number of subsets is finite and $S_J>0$, choose
\[
0<\varepsilon<\min_{\varnothing\ne J\subseteq I}
\frac{8\pi S_J-Q_J}{8\pi S_J}\le1.
\]
Then $Q_J/(1-\varepsilon)<8\pi S_J$. Set $\widetilde A=A/(1-\varepsilon)$ and apply Lemma~\ref{lem:HLS} to this matrix. Since $W_{\widetilde A}=W_A/(1-\varepsilon)$,
\[
H-W_A=\varepsilon H+(1-\varepsilon)(H-W_{\widetilde A})\ge\varepsilon H-C.
\]
The external-potential energy in \eqref{eq:energy} has absolute value at most $\sum_i\rho_i\norm{\log(\rho_i h_i)}_\infty$.
\end{proof}

We apply the following estimates to the maximal solution constructed in Lemma~\ref{lem:local}.

\begin{lemma}\label{lem:smallmass}

Assume \eqref{hyp:A} and $\Lambda_J(\rho)>0$ for every nonempty $J\subseteq I$. Along the maximal solution from Lemma~\ref{lem:local}, $H(c(t))\le C$ uniformly for $t<T_{\max}$. For every $\delta>0$ there exists $r_\delta>0$ such that
\begin{equation}\label{eq:smallmass}
\sup_{0\le t<T_{\max}}\sup_{x\in M}
\sum_i\int_{B_{2r_\delta}(x)}c_i(t)\dV<\delta.
\end{equation}
\end{lemma}
\begin{proof}

From \eqref{eq:dissipation} and \eqref{eq:coercivity}, we obtain
\[
\varepsilon H(c(t))\le\F(c(t))+C\le\F(c_0)+C.
\]
Since $s\log s\ge-e^{-1}$ for $s\ge0$, there exists a constant $C=C(M,g,A,\rho,h,c_0)>0$, independent of $t$, such that 
\[
\sum_i\int_M(c_i(t)\log c_i(t))_+\dV\le C,
\qquad 0\le t<T_{\max},
\]
where $z_+=\max\{z,0\}$. For every measurable $E\subseteq M$, every $K>e$ and every $0\le t<T_{\max}$, it follows that

\[
\sum_i\int_E c_i\dV
\le nK\operatorname{vol}_g(E)+\frac1{\log K}
\sum_i\int_{\{c_i>K\}}c_i\log c_i\dV
\le nK\operatorname{vol}_g(E)+\frac C{\log K}.
\]

Choose $K>e$ so that $C/\log K<\delta/2$. There exist $r_*>0$ and $C_V>0$, depending only on $(M,g)$, such that
\[
\operatorname{vol}_g(B_s(x))\le C_Vs^2
\qquad(x\in M,\ 0<s\le r_*).
\]
Choose $0<r_\delta<\min\{r_*/2,\sqrt{\delta/(8nKC_V)}\}$. Then

\[
nK\sup_x\operatorname{vol}_g(B_{2r_\delta}(x))<\delta/2.
\]
\end{proof}

The following local interpolation estimate is independent of the subcriticality assumption.
\begin{lemma}\label{lem:GN}

Fix $p>1$. Let $B_{2r}=B_{2r}(x_0)$ be a sufficiently small geodesic ball and let $f,\eta\in C_c^\infty(B_{2r})$ satisfy $f\ge0$ and $0\le\eta\le1$. Define $v=\eta f^{p/2}$. Then, we get
\begin{equation}\label{eq:GN}
\int v^{2(p+1)/p}\dV\le C_p\left(\int_{B_{2r}}f\dV\right)
\int|\nabla_g v|^2\dV,
\end{equation}
with $C_p$ independent of the small radius and the center.
\end{lemma}
\begin{proof}
For $w\in C_c^1(\mathbb R^2)$, the fundamental theorem of calculus in each coordinate gives
\[
|w(x_1,x_2)|\le\int_{\mathbb R}|\partial_1w(s,x_2)|\,\mathrm ds,
\qquad
|w(x_1,x_2)|\le\int_{\mathbb R}|\partial_2w(x_1,s)|\,\mathrm ds.
\]
Multiplication and integration yield $\norm w_2^2\le\norm{\partial_1w}_1\norm{\partial_2w}_1\le\norm{\nabla w}_1^2$.

For $w\in W^{1,1}_0(\mathbb R^2)$, choose $w_k\in C_c^\infty(\mathbb R^2)$ such that $w_k\to w$ in $W^{1,1}$. After passing to a subsequence, $w_k\to w$ almost everywhere. Fatou's lemma and $\nabla w_k\to\nabla w$ in $L^1$ give
\[
\norm w_2^2\le\liminf_{k\to\infty}\norm{w_k}_2^2
\le\lim_{k\to\infty}\norm{\nabla w_k}_1^2
=\norm{\nabla w}_1^2.
\]
We explain why its constant remains uniform in geodesic coordinates. Choose $0<r_*<\operatorname{inj}(M,g)$ sufficiently small. Since $M$ is compact and the metric is smooth, normal coordinates based on an orthonormal frame at any $x_0\in M$ satisfy
\[
\frac12|\xi|^2\le g_{ab}(y)\xi^a\xi^b\le2|\xi|^2,
\qquad \frac12\le\sqrt{\det g(y)}\le2
\quad(|y|<r_*).
\]
The constants are independent of $x_0$. Indeed, $g_{ab}(0)=\delta_{ab}$ and the coordinate coefficients depend smoothly on the base point, the orthonormal frame and $y$; compactness of the orthonormal frame bundle gives a common radius for these bounds. The inverse matrix has the same eigenvalue bounds. For $2r\le r_*$ and $\widehat w=w\circ\exp_{x_0}$, extended by zero to $\mathbb R^2$, we have
\[
\int_{B_{2r}}|w|^2\dV\le2\int_{\mathbb R^2}|\widehat w|^2\,dy,
\qquad
\int_{\mathbb R^2}|D\widehat w|\,dy
\le2\sqrt2\int_{B_{2r}}|\nabla_gw|\dV.
\]
Consequently the Euclidean inequality gives
\[
\int_{B_{2r}}|w|^2\dV
\le16\left(\int_{B_{2r}}|\nabla_gw|\dV\right)^2,
\]
uniformly in $r$ and $x_0$. The same volume-density bounds give $\operatorname{vol}_g(B_s(x_0))\le2\pi s^2$ for $s\le r_*$. Apply the preceding Sobolev inequality to $w=v^{1+1/p}$ and use Cauchy--Schwarz to obtain

\[
\int v^{2(p+1)/p}\dV
\le C_p\left(\int v^{1/p}|\nabla_gv|\dV\right)^2
\le C_p\int v^{2/p}\dV\int|\nabla_gv|^2\dV.
\]

The required Sobolev regularity is valid also for noninteger $p>1$. For the nonnegative smooth compactly supported $f$, Taylor's theorem in coordinates gives $|Df|^2\le C_f f$. Thus $|D(f^{p/2})|\le C_{p,f}f^{(p-1)/2}$ on $\{f>0\}$. Regularizing by $(f+\varepsilon)^{p/2}$ and passing to the limit shows that $v\in H^1_0(B_{2r})$. To justify the preceding chain rule, first use $v_N=\min\{v,N\}$.  The Sobolev chain rule places $v_N^{1+1/p}$ in $H^1_0\subset W^{1,1}_0$, while $v_N^{2/p}\le v^{2/p}\le f$ and $|\nabla v_N|\le|\nabla v|$. Monotone convergence gives \eqref{eq:GN}. This argument integrates powers of nonnegative functions and remains valid when $2/p<1$.
\end{proof}

\begin{proposition}\label{prop:Lp}
Assume \eqref{hyp:A} and $\Lambda_J(\rho)>0$ for every nonempty $J\subseteq I$, and let $c$ be the maximal solution from Lemma~\ref{lem:local}.

For every fixed $p>1$, the smooth positive solution satisfies
\[
\sup_{t<T_{\max}}\sum_i\int c_i(t)^p\dV\le C_p.
\]
\end{proposition}
\begin{proof}
\par\smallskip
\noindent\emph{Step 1. The differential inequality.}
Write $b_i=\nabla_g\log h_i+\nabla_gU_i[c]$, so that $\partial_tc_i=\Delta_gc_i-\Div(c_ib_i)$ and $\Div b_i=\Delta_g\log h_i-\sum_j a_{ij}(c_j-\rho_j)$. Set
\[
Y_i=\int c_i^p\dV,\quad D_i=\int|\nabla_gc_i^{p/2}|^2\dV,
\quad Y=\sum_iY_i,\quad D_p=\sum_iD_i.
\]
Strict positivity on compact time intervals justifies testing by $c_i^{p-1}$ even for noninteger $p$. Integration by parts gives
\[
\int_M c_i^{p-1}\Delta_gc_i\dV
=-(p-1)\int_Mc_i^{p-2}|\nabla_gc_i|^2\dV
=-\frac{4(p-1)}{p^2}D_i.
\]
For the drift term, we have

\[
-\int c_i^{p-1}\Div(c_ib_i)\dV
=\frac{p-1}{p}\int\nabla_g(c_i^p)\cdot b_i\dV
=-\frac{p-1}{p}\int c_i^p\Div b_i\dV.
\]
Since $a_{ij}\rho_j\ge0$, we get
\begin{equation}\label{eq:Lp-test}
Y'+\frac{4(p-1)}pD_p
\le C_{p,h}Y+(p-1)\sum_{i,j}a_{ij}\int c_i^pc_j\dV.
\end{equation}

\par\smallskip
\noindent\emph{Step 2. Absorption on small balls.}

Choose $\delta>0$, to be fixed below, and decrease the radius $r_\delta$ from Lemma~\ref{lem:smallmass} so that $r=r_\delta$ satisfies
\[
\frac{5r}{2}\le r_*,\qquad
c_Vs^2\le\operatorname{vol}_g(B_s(x))\le C_Vs^2
\quad(x\in M,\ 0<s\le r_*),
\]
where $r_*$ is the common normal-coordinate radius from the proof of Lemma~\ref{lem:GN}. The metric bounds there hold for $|y|<r_*$. We use a standard covering argument for balls of equal radius. See, for example, \cite[Chapter~I, Section~7.1, p.~32]{SteinHA}. Choose a maximal finite set $\{x_\ell\}_{\ell=1}^N$ satisfying
\[
d_g(x_\ell,x_k)\ge r\quad(\ell\ne k).
\]
Such a set is obtained by successively adding points at distance at least $r$ from all previous centers. Compactness makes this process finite, and maximality gives
\[
M=\bigcup_{\ell=1}^N B_r(x_\ell),\qquad
B_{r/2}(x_\ell)\cap B_{r/2}(x_k)=\varnothing\quad(\ell\ne k).
\]
For fixed $x\in M$, put $N(x)=\#\{\ell\mid x\in B_{2r}(x_\ell)\}$. For every index counted in $N(x)$, $B_{r/2}(x_\ell)\subset B_{5r/2}(x)$. Hence
\[
N(x)c_V(r/2)^2\le\operatorname{vol}_g(B_{5r/2}(x))
\le C_V(5r/2)^2,
\qquad
N(x)\le N_0:=\left\lceil\frac{25C_V}{c_V}\right\rceil.
\]
Here $\lceil x\rceil$ denotes the smallest integer greater than or equal to $x$. Thus the enlarged balls have uniformly bounded overlap. Choose disjoint measurable sets $E_\ell\subset B_r(x_\ell)$ covering $M$ and cutoffs $\eta_\ell$ equal to one on $B_r(x_\ell)$, supported in $B_{2r}(x_\ell)$, with $|\nabla\eta_\ell|\le C/r$.

For each $\ell$, choose $\chi_\ell\in C_c^\infty(B_{2r}(x_\ell))$ with $0\le\chi_\ell\le1$ and $\chi_\ell=1$ on a neighborhood of $\operatorname{supp}\eta_\ell$. Apply Lemma~\ref{lem:GN} to $f=\chi_\ell c_i$ and $\eta=\eta_\ell$. Then $\eta_\ell f^{p/2}=\eta_\ell c_i^{p/2}$ and
$\int_{B_{2r}(x_\ell)}f\dV\le\int_{B_{2r}(x_\ell)}c_i\dV<\delta$. Young's inequality and \eqref{eq:GN} therefore give

\[
c_i^pc_j\le\frac p{p+1}c_i^{p+1}+\frac1{p+1}c_j^{p+1},\qquad
\int_{E_\ell}c_i^{p+1}\dV\le C_p\delta\int|\nabla_g(\eta_\ell c_i^{p/2})|^2\dV.
\]
Expanding the cutoff derivative and using bounded overlap,
\[
\sum_\ell\int|\nabla_g(\eta_\ell c_i^{p/2})|^2\dV
\le C D_i+Cr^{-2}Y_i.
\]
Let $R_A=\max_i\sum_j a_{ij}$. Symmetry gives the same bound for column sums. Consequently,
\[
\sum_{i,j}a_{ij}\int c_i^pc_j\dV
\le R_A\sum_i\int c_i^{p+1}\dV
\le C_{p,A}\delta D_p+C_{p,A}\delta r^{-2}Y.
\]
The coefficient of $\delta D_p$ is independent of $r$, $\delta$ and the number of balls. Choose $\delta$ so that $(p-1)C_{p,A}\delta\le2(p-1)/p$ and then fix $r=r_\delta$. We obtain
\begin{equation}\label{eq:Lp-absorb}
Y'+\kappa_pD_p\le C_pY,\qquad \kappa_p=2(p-1)/p>0.
\end{equation}

If $A=0$, then \(\sum_{i,j}a_{ij}\int_Mc_i^pc_j\dV=0\), and \eqref{eq:Lp-absorb} follows directly from \eqref{eq:Lp-test}.

\par\smallskip
\noindent\emph{Step 3. A bound uniform in time.}

Since $D_p\ge0$, \eqref{eq:Lp-absorb} and Gronwall's inequality give
\[
Y(t)\le Y(0)e^{C_pt}\qquad(0\le t<T_{\max}).
\]
In particular, for every finite $T>0$, $\sup_{0\le t<\min\{T,T_{\max}\}}Y(t)\le Y(0)e^{C_pT}$. This excludes finite-time divergence of $Y$ on the interval of existence. Global existence of the solution will follow from the supremum estimate and continuation argument in Proposition~\ref{prop:regularity}. To obtain a bound independent of time, apply the cutoff estimate once more with a fixed finite cover and the global mass bound. It yields $\int c_i^{p+1}\dV\le C_{p,\rho_i}(D_i+Y_i)$. H\"older's inequality gives
\[
Y_i=\int c_i^{1/p}(c_i^{p+1})^{(p-1)/p}\dV
\le\rho_i^{1/p}\left(\int c_i^{p+1}\dV\right)^{(p-1)/p}.
\]
With $r_p=p/(p-1)>1$, this implies $Y_i^{r_p}\le C(D_i+Y_i)$. The inequality $Cy\le y^{r_p}/2+C'$ for $y\ge0$, obtained by splitting at $(2C)^{1/(r_p-1)}$, absorbs the last term. Thus
\[
Y^{r_p}\le n^{r_p-1}\sum_iY_i^{r_p}\le C(1+D_p).
\]
Substituting in \eqref{eq:Lp-absorb} gives $Y'+a_pY^{r_p}\le b_p(1+Y)$ for fixed positive $a_p,b_p$. Choose
\[
R\ge\max\{1,(4b_p/a_p)^{1/(r_p-1)}\}.
\]
The derivative is negative whenever $Y\ge R$. Scalar comparison proves $Y(t)\le\max\{Y(0),R\}$ throughout the interval of existence.
\end{proof}

\begin{proposition}\label{prop:regularity}
Assume \eqref{hyp:A} and $\Lambda_J(\rho)>0$ for every nonempty $J\subseteq I$, and let $c$ be the maximal solution from Lemma~\ref{lem:local}. The maximal time is infinite and \eqref{eq:high} holds.

\end{proposition}
\begin{proof}
\par\smallskip
\noindent\emph{Step 1. Global existence and regularity.}
Take $p=3$ in Proposition~\ref{prop:Lp}. The elliptic estimate and the embedding $W^{2,3}\hookrightarrow C^{1,1/3}$ give $\norm{U_i[c(t)]}_{W^{2,3}}+\norm{b_i(t)}_\infty\le C$. Thus $\norm{c_i(t)b_i(t)}_3\le C$. For $0<\ell\le\min\{1,t\}$, Duhamel's formula and \eqref{eq:heat} yield
\[
\begin{split}
\norm{c_i(t)}_\infty
&\le\norm{S(\ell)c_i(t-\ell)}_\infty+
\int_{t-\ell}^t\norm{S(t-s)\Div(c_i(s)b_i(s))}_\infty\,\mathrm ds\\
&\le C\ell^{-1/3}+C\int_0^\ell s^{-5/6}\,\mathrm ds
=C\ell^{-1/3}+6C\ell^{1/6}.
\end{split}
\]
For $t\ge\tau>0$, choose $\ell=\min\{1,\tau\}$. If $T_{\max}<\infty$, the choice $\tau=T_{\max}/2$ contradicts Lemma~\ref{lem:local}. Hence $T_{\max}=\infty$. Together with continuity near $t=0$, the preceding estimate gives
\[
c\in C([0,\infty);L^\infty(M)^n),\qquad
\sup_{t\ge0}\|c(t)\|_{L^\infty(M)^n}<\infty.
\]

Fix $\tau>0$ and put $\ell=\min\{1/4,\tau/16\}$. For $t_0\ge\tau$, apply Lemma~\ref{lem:smoothing} on $[t_0-2\ell,t_0+\ell]$, with $L=3\ell$ and $\delta=\ell$. These intervals lie in $[\tau/2,\infty)$, where the preceding supremum bound is uniform. The lemma therefore gives all the bounds in the first line of \eqref{eq:high}, with constants independent of $t_0$.

\par\smallskip
\noindent\emph{Step 2. A uniform positive lower bound.}
We prove the positive lower bound by compactness and the strong maximum principle. Otherwise, after restricting to a fixed component and passing to a subsequence, there would exist $t_k\to\infty$ and $x_k\to x_*$ with $c_i(t_k,x_k)\to0$. Set $v^k(s,x)=c(t_k+s,x)$ for $-1\le s\le1$. The preceding parabolic bounds and Arzel\`a--Ascoli give convergence on $[-3/4,3/4]\times M$ in $C^{2,1}$ to a nonnegative solution $v$. Elliptic regularity passes the potentials and their derivatives to the limit, while uniform convergence preserves every mass. In particular,
\[
v_i(0,x_*)=0,\qquad \int_Mv_i(s)\dV=\rho_i>0.
\]
For $Q\ge\sup|q_i[v]|$, the function $z=e^{-Qs}v_i$ satisfies
\[
z_s-\Delta_gz+b_i[v]\cdot\nabla_gz+(q_i[v]+Q)z=0,
\qquad q_i[v]+Q\ge0.
\]
Apply the strong maximum principle \cite[Theorem 2.7]{Lieberman} to $-z\le0$, whose maximum zero is attained at an interior point. Connectedness forces $z$ to vanish on the earlier slice $s=-1/2$, contradicting its positive mass. On the remaining compact time interval, continuity and strict positivity provide a positive minimum. This proves \eqref{eq:high}.
\end{proof}

\begin{proposition}\label{prop:subsequence}
Assume \eqref{hyp:A} and $\Lambda_J(\rho)>0$ for every nonempty $J\subseteq I$, and let $c$ be the maximal solution from Lemma~\ref{lem:local}.

There exist $t_k\to\infty$ and a smooth positive stationary state $c^*$ such that $c(t_k)\to c^*$ in $C^m(M)^n$ for every integer $m\geq0$.
\end{proposition}
\begin{proof}
Jensen's inequality gives $\int c_i\log c_i\dV\ge\rho_i\log\rho_i$. Thus \eqref{eq:coercivity} bounds the energy from below. Equation~\eqref{eq:dissipation} implies
\[
\int_0^\infty\D(c(t))\,\mathrm dt
=\F(c_0)-\lim_{t\to\infty}\F(c(t))<\infty.
\]
Choose $t_k\in[k,k+1]$ so that $\D(c(t_k))\le2\int_k^{k+1}\D(c(t))\,\mathrm dt$. Such a choice follows by integration when the integral is positive. When it vanishes, continuity permits any point of the interval. Hence $\D(c(t_k))\to0$. The uniform $C^{m+1}$ bounds give compactness in $C^m$. We get a subsequence converging in $C^m(M)^n$ for every integer $m\ge0$, by a diagonal selection. The limits agree by uniqueness of the uniform limit. Mass and the positive lower bound pass to $c^*$. Moreover, $\norm{\nabla_g\mu(c(t_k))}_2^2\le C\D(c(t_k))\to0$. The $C^1$ convergence, the lower bound and elliptic regularity imply uniform convergence of $\nabla\log c_i(t_k)$ and $\nabla U_i[c(t_k)]$. Therefore $\nabla\mu_i(c^*)=0$. By Lemma~\ref{lem:stationary},
\[
c_i^*=\rho_i\frac{h_i e^{U_i[c^*]}}{\int_Mh_i e^{U_i[c^*]}\dV},
\qquad
\Div\left(c_i^*\nabla_g\left(\log\frac{c_i^*}{h_i}-U_i[c^*]\right)\right)=0
\quad(i\in I).
\]
 Energy continuity along this subsequence identifies the decreasing energy limit with $\F(c^*)$.
\end{proof}

\Needspace{8\baselineskip}
\section{Convergence to a stationary density}\label{sec:convergence}
Under the assumptions of Theorem~\ref{thm:main}, we prove that the entire trajectory converges to the stationary density obtained in Proposition~\ref{prop:subsequence}. We first establish a local gradient inequality near a smooth positive stationary density and then use it with the energy identity and the uniform estimates from Section~\ref{sec:global}. 

\subsection{A gradient inequality}
\begin{lemma}\label{lem:LS}

Let $c^*$ be a smooth positive stationary density with $\int_Mc_i^*\dV=\rho_i$ and let $0<\alpha<1$. There exist $\theta\in(0,1/2]$, $C>0$ and $\sigma>0$ such that every $c\in C^\infty(M)^n$ belonging to
\[
\mathcal U=\left\{c\in C^{0,\alpha}(M)^n\ \middle|\
c_i>0,\ \int_Mc_i\dV=\rho_i\ (i\in I),\
\|c-c^*\|_{C^{0,\alpha}}<\sigma\right\}
\]
satisfies

\begin{equation}\label{eq:LS}
|\F(c)-\F(c^*)|^{1-\theta}\le C\norm{P\mu(c)}_2,
\qquad (Pf)_i=f_i-\int_M f_i\dV,\quad i=1,\ldots,n
\end{equation}
and $\mu$ is defined as in (\ref{def:mu}).
\end{lemma}
\begin{proof}
\par\smallskip
\noindent\emph{Step 1. The linearized operator.}
Let 
\[
\Hh
:=
\left\{
\xi=(\xi_1,\ldots,\xi_n)\in L^2(M)^n
\;\middle|\;
\int_M \xi_i\,dV_g=0
\quad\text{for every }i=1,\ldots,n
\right\}
\]
and let $X=C^{0,\alpha}(M)^n\cap\Hh$, equipped with the H\"older norm. For $\xi\in X$ sufficiently small, set
\[
\mathcal M(\xi):=P\mu(c^*+\xi)
=P\left(
\log\frac{c_i^*+\xi_i}{\rho_i h_i}+1
-\sum_j a_{ij}\G(c_j^*+\xi_j-\rho_j)
\right)_{i\in I}.
\]
This defines a map from a neighborhood of $0$ in $X$ into $X$. Since $c^*$ is stationary, $\mathcal M(0)=0$. Then, $D\F(c^*+\xi)[\zeta]=\langle\mathcal M(\xi),\zeta\rangle$ for $\zeta\in X$, where the inner product is that of $L^2$. The linearization is
\begin{equation}\label{eq:Hessian}
L\xi=P\left[\left(\frac{\xi_i}{c_i^*}\right)_i-
\left(\sum_j a_{ij}\G\xi_j\right)_i\right]=L_0\xi-K\xi.
\end{equation}
Here
\[
 (L_0\xi)_i=P_i(\xi_i/c_i^*),\qquad
 (K\xi)_i=\sum_j a_{ij}\mathcal G\xi_j.
\]
$L_0$ is a linear operator from $\Hh$ into itself. The Green potentials have zero mean, so $K$ also maps $\Hh$ into itself. On both $\Hh$ and $X$, the operator $L_0$ has the bounded inverse
\begin{equation}\label{eq:inverse}
(L_0^{-1}f)_i=c_i^*\left(f_i-\frac1{\rho_i}\int c_i^*f_i\dV\right).
\end{equation}
Indeed, for $f\in\Hh$, the vector $\xi$ defined by the right-hand
side of \eqref{eq:inverse} satisfies
\[
\begin{aligned}
\int_M\xi_i\dV
&=\int_Mc_i^*f_i\dV
-\frac{\int_Mc_i^*f_i\dV}{\rho_i}\int_Mc_i^*\dV=0,\\
P_i(\xi_i/c_i^*)
&=P_i\left(f_i-\frac1{\rho_i}\int_Mc_i^*f_i\dV\right)
=f_i.
\end{aligned}
\]
Moreover, $L_0\xi=0$ implies $\xi_i=\lambda_i c_i^*$, and $0=\int_M\xi_i\dV=\lambda_i\rho_i$ forces $\lambda_i=0$. Thus \eqref{eq:inverse} is indeed the inverse of $L_0$. The Green operator maps $L^2_0$ boundedly into $H^2$, which embeds compactly into $L^2$. Hence $K$ is compact on $\Hh$ and $L=L_0(1-L_0^{-1}K)$ is Fredholm of index zero. It is self-adjoint, since $A$ and $G$ are symmetric and $P$ can be omitted when paired against a zero-mean vector. With
\[
\Kk=\ker L,\qquad \Rr=\Kk^\perp,
\]
self-adjointness gives
$(\operatorname{ran}L)^\perp=\ker L$, and closedness of the range therefore yields $\operatorname{ran}L=\Rr$. Thus $L:\Rr\to\Rr$ is a bounded bijection, and the bounded inverse theorem gives
\begin{equation}\label{eq:L-inverse-bound}
\norm v_2\le C_L\norm{Lv}_2\qquad(v\in\Rr).
\end{equation}

By \eqref{eq:Hessian}, every $\xi\in\Kk$ satisfies
\[
\frac{\xi_i}{c_i^*}
=\lambda_i+\sum_j a_{ij}\G\xi_j,
\qquad \lambda_i\in\mathbb R,\quad i\in I.
\]
Elliptic regularity, starting from $L^2$, shows that every kernel element is smooth. Let $\{e_\ell\}_{\ell=1}^d$ be a smooth $L^2$-orthonormal basis of $\Kk$, where $d=\dim\Kk$. Define the orthogonal projection by
\[
\Pi_\Kk:\Hh\longrightarrow\Kk,\qquad
\Pi_\Kk v:=\sum_{\ell=1}^d\langle v,e_\ell\rangle_{L^2}e_\ell.
\]
Since each $e_\ell$ is smooth, its restriction to $X$ is bounded. The projection $\Pi_\Rr=1-\Pi_\Kk$ is also bounded on $X$. For $f\in\Rr\cap X$, solve $Lw=f$ in $\Rr$ and write $w=L_0^{-1}(f+Kw)$. Since $H^2(M)\hookrightarrow C^{0,\alpha}(M)$ in dimension two for $\alpha<1$,
\[
\norm w_{C^{0,\alpha}}\le C\bigl(\norm f_{C^{0,\alpha}}+\norm w_2\bigr)
\le C\norm f_{C^{0,\alpha}}.
\]
Hence $L:\Rr\cap X\to\Rr\cap X$ is also an isomorphism.

\par\smallskip
\noindent\emph{Step 2. Finite-dimensional reduction.}
The H\"older space is a Banach algebra. After reducing the neighborhood, the series
\[
\log(c_i^*+\xi_i)=\log c_i^*+
\sum_{k\ge1}\frac{(-1)^{k+1}}k\left(\frac{\xi_i}{c_i^*}\right)^k
\]
converges in that algebra. Thus $\mathcal M$ and $\F(c^*+\xi)$ are real analytic in $\xi$. Apply the analytic implicit function theorem to $\Pi_\Rr\mathcal M(z+w)=0$ with $z\in\Kk$ and $w\in\Rr\cap X$. It gives an analytic map $\Psi$ near zero, with $\Psi(0)=0$ and
\begin{equation}\label{eq:center}
\Pi_\Rr\mathcal M(z+\Psi(z))=0.
\end{equation}
Set $f(z)=\F(c^*+z+\Psi(z))$. For $k\in\Kk$, we define $D\Psi(z)[k]=\left.\frac{\mathrm d}{\mathrm d\varepsilon}\Psi(z+\varepsilon k)\right|_{\varepsilon=0}$. This vector lies in $\Rr$, so
\[
Df(z)[k]=\langle\mathcal M(z+\Psi(z)),k+D\Psi(z)[k]\rangle
=\langle\Pi_\Kk\mathcal M(z+\Psi(z)),k\rangle.
\]
The finite-dimensional analytic gradient inequality \cite[Theorem 2.9(1)]{FM} gives an exponent $a_0\in(0,1)$. Shrink the neighborhood so that $|f-f(0)|\le1$ and take $a=\max\{a_0,1/2\}$. Then
\begin{equation}\label{eq:finite}
|f(z)-f(0)|^a\le C\norm{\Pi_\Kk\mathcal M(z+\Psi(z))}_2,
\qquad \tfrac12\le a<1.
\end{equation}
If $\Kk=\{0\}$, the left side is zero and we take $a=1/2$.

\par\smallskip
\noindent\emph{Step 3. The error in the $L^2$ norm.}
Decompose $\xi=z+w$ and set
\[
\xi_0=z+\Psi(z),\qquad
\widetilde w=w-\Psi(z)\in\Rr,\qquad
\xi_s=\xi_0+s\widetilde w=(1-s)\xi_0+s\xi.
\]
Shrink the neighborhood so that $c_i^*+\xi_{s,i}\ge\frac12\min_M c_i^*>0$ for $0\le s\le1$. By the definition of $U$ and the linearity of $\G$,
\[
-U_i[c^*+\xi]+U_i[c^*+\xi_0]
=-\sum_j a_{ij}\G\widetilde w_j
=-(K\widetilde w)_i.
\]
For the logarithmic term, the fundamental theorem of calculus gives
\[
\log\frac{c_i^*+\xi_i}{c_i^*+\xi_{0,i}}
=\widetilde w_i\int_0^1\frac{\mathrm ds}{c_i^*+\xi_{s,i}}.
\]
Separating the linear term $\widetilde w_i/c_i^*$ and combining it with the Green-potential term, we obtain
\[
\mathcal M(\xi)-\mathcal M(\xi_0)=L\widetilde w+r,
\qquad
r=P\left(
\widetilde w_i\int_0^1
\left(\frac1{c_i^*+\xi_{s,i}}-\frac1{c_i^*}\right)
\,\mathrm ds
\right)_{i\in I}.
\]
Using the positive lower bound and the $L^2$ contractivity of $P$, we obtain
\[
\left|\frac1{c_i^*+\xi_{s,i}}-\frac1{c_i^*}\right|
=\frac{|\xi_{s,i}|}{c_i^*(c_i^*+\xi_{s,i})}
\le C|\xi_{s,i}|,
\qquad
\norm r_2
\le C(\norm\xi_\infty+\norm{\xi_0}_\infty)
\norm{\widetilde w}_2.
\]
Reduce the neighborhood until this last coefficient, denoted by $\eta$, satisfies $\eta C_L\le1/2$. Since $\Pi_\Rr\mathcal M(\xi_0)=0$ and $L\widetilde w\in\Rr$, projection onto $\Rr$ and \eqref{eq:L-inverse-bound} yield
\[
\begin{aligned}
L\widetilde w&=\Pi_\Rr\mathcal M(\xi)-\Pi_\Rr r,\\
\norm{\widetilde w}_2
&\le C_L\norm{\mathcal M(\xi)}_2
+C_L\eta\norm{\widetilde w}_2,
\end{aligned}
\qquad
\norm{\widetilde w}_2\le2C_L\norm{\mathcal M(\xi)}_2.
\]
Projection onto $\Kk$ removes $L\widetilde w$ and gives
\begin{equation}\label{eq:reducedgradient}
\begin{aligned}
\Pi_\Kk\mathcal M(\xi_0)
&=\Pi_\Kk\mathcal M(\xi)-\Pi_\Kk r,\\
\norm{\Pi_\Kk\mathcal M(\xi_0)}_2
&\le(1+2C_L\eta)\norm{\mathcal M(\xi)}_2
\le2\norm{\mathcal M(\xi)}_2.
\end{aligned}
\end{equation}
At the reduced point, $\mathcal M(\xi_0)\in\Kk$ and $\widetilde w\in\Rr=\Kk^\perp$, so
\[
D\F(c^*+\xi_0)[\widetilde w]
=\langle\mathcal M(\xi_0),\widetilde w\rangle_{L^2}=0.
\]
Moreover,
\[
D^2\F(c)[v,v]
=\sum_i\int\frac{v_i^2}{c_i}\dV
-\sum_{i,j}a_{ij}\int v_i\G v_j\dV.
\]
The uniform lower bound along the segment and boundedness of $\G:L^2_0\to L^2_0$ imply
$|D^2\F(c^*+\xi_s)[v,v]|\le C\norm v_2^2$. Since $f(z)=\F(c^*+\xi_0)$, Taylor's formula therefore gives
\[
\begin{aligned}
\F(c^*+\xi)-f(z)
&=\int_0^1(1-s)D^2\F(c^*+\xi_s)
[\widetilde w,\widetilde w]\,\mathrm ds,\\
|\F(c^*+\xi)-f(z)|
&\le C\norm{\widetilde w}_2^2
\le C\norm{\mathcal M(\xi)}_2^2.
\end{aligned}
\]
Recall that $f(0)=\F(c^*)$ and $1/2\le a<1$. Combining this estimate with \eqref{eq:finite} and \eqref{eq:reducedgradient}, and shrinking the neighborhood so that $\norm{\mathcal M(\xi)}_2\le1$, yields
\[
\begin{aligned}
|\F(c^*+\xi)-\F(c^*)|^a
&\le|\F(c^*+\xi)-f(z)|^a+|f(z)-f(0)|^a\\
&\le C\norm{\mathcal M(\xi)}_2^{2a}
+C\norm{\Pi_\Kk\mathcal M(\xi_0)}_2\\
&\le C\left(\norm{\mathcal M(\xi)}_2^{2a}
+\norm{\mathcal M(\xi)}_2\right)
\le C\norm{\mathcal M(\xi)}_2,
\end{aligned}
\]
where the last inequality uses $2a\ge1$. Set $\theta=1-a$. All analyticity used above concerns the density variable in $X$.
\end{proof}

\subsection{Whole-trajectory convergence and rates}
\begin{proposition}\label{prop:whole}
The state $c^*$ in Proposition~\ref{prop:subsequence} is the limit of the whole trajectory in $C^m(M)^n$ for every integer $m\geq0$.
\end{proposition}
\begin{proof}
\par\smallskip
\noindent\emph{Step 1. Length in the gradient neighborhood.}
For all sufficiently large times, fix $0<m_0\le c_i(t,x)\le M_0$. Poincar\'e's inequality implies
\begin{align}
\norm{P\mu(c(t))}_2^2
\le C\sum_i\norm{\nabla_g\mu_i(c(t))}_2^2\le \frac{C}{m_0}\sum_i\int_M
c_i(t,x)|\nabla_g\mu_i(c(t))|^2\dV=\frac{C}{m_0}\D(c(t)).\nonumber
\end{align}
Let $E(t)=\F(c(t))-\F(c^*)$. Proposition~\ref{prop:subsequence} identifies the energy limit, so $E(t)\ge0$ and $E(t)\to0$. In the neighborhood of Lemma~\ref{lem:LS}, we get
\begin{equation}\label{eq:LSdiss}
E(t)^{1-\theta}\le C\D(c(t))^{1/2}.
\end{equation}
If $E(t_0)=0$ at a finite time, the energy is constant thereafter. Continuity and the dissipation identity give $\D=0$, so the solution is stationary. Its subsequential limit identifies it with $c^*$. We may therefore assume $E(t)>0$ at every finite time. In the same neighborhood, \eqref{eq:LSdiss} implies $\D>0$ and
\[
-\frac{\mathrm d}{\mathrm dt}E^\theta
=\theta\frac{\D}{E^{1-\theta}}\ge\frac\theta C\D^{1/2}.
\]
Use $H^{-1}(M)=(H^1(M))^*$. Testing the equation against $\varphi\in H^1$, we get
\[
|\langle\partial_tc_i,\varphi\rangle|
\le\left(\int c_i|\nabla_g\mu_i|^2\dV\right)^{1/2}
M_0^{1/2}\norm\varphi_{H^1}.
\]
Consequently, as long as the trajectory remains in the neighborhood on $[T,S]$, we get
\begin{equation}\label{eq:length}
\int_T^S\norm{\partial_tc}_{H^{-1}}\,\mathrm dt
\le C\bigl(E(T)^\theta-E(S)^\theta\bigr)\le CE(T)^\theta.
\end{equation}

\par\smallskip
\noindent\emph{Step 2. The first-exit argument.}
We prove that the trajectory remains in $\mathcal U$ after some time. Choose $s>3+\alpha$ and an integer $N>s$. With $\beta=(N-s)/(N+1)$, by Sobolev embedding and interpolation, we get
\begin{equation}\label{eq:trap}
\norm v_{C^{2,\alpha}}\le C\norm v_{H^s}
\le C\norm v_{H^{-1}}^\beta\norm v_{H^N}^{1-\beta}.
\end{equation}
Indeed, in an orthonormal eigenbasis of $-\Delta_g$, the identity $s=-\beta+(1-\beta)N$ gives
\[
\sum_\ell(1+\lambda_\ell)^s|v_\ell|^2
=\sum_\ell\bigl((1+\lambda_\ell)^{-1}|v_\ell|^2\bigr)^\beta
\bigl((1+\lambda_\ell)^N|v_\ell|^2\bigr)^{1-\beta}.
\]
H\"older's inequality with exponents $1/\beta$ and $1/(1-\beta)$ proves the interpolation estimate (\ref{eq:trap}).

Let $B_N$ bound $\norm{c(t)}_{H^N}$ at large times. Choose $\sigma>0$ so that the closed $C^{2,\alpha}$ ball of radius $\sigma$ about $c^*$ in the fixed-mass class consists of positive densities and lies in $\mathcal U$. Along the subsequence, take $T$ sufficiently large that
\begin{align}\label{ineq:CT<sigma/4}
\norm{c(T)-c^*}_{C^{2,\alpha}}<\sigma/4,
\qquad C(2B_N)^{1-\beta}E(T)^{\beta\theta}<\sigma/4.
\end{align}
The first condition follows from convergence in $C^3$, hence in $C^{2,\alpha}$. We define the first exit time
\[
S=\inf\{t\ge T:\norm{c(t)-c^*}_{C^{2,\alpha}}\ge\sigma\},
\qquad\inf\varnothing=\infty.
\]
By \eqref{ineq:CT<sigma/4} and continuity, $S>T$. Suppose that $S<\infty$. Then, $\norm{c(S)-c^*}_{C^{2,\alpha}}=\sigma$. For $T\le t<S$, the trajectory remains in $\mathcal U$, so \eqref{eq:length} gives
\[
\norm{c(t)-c(T)}_{H^{-1}}
\le\int_T^t\norm{\partial_sc(s)}_{H^{-1}}\,\mathrm ds
\le CE(T)^\theta.
\]
Since $\norm{c(t)-c(T)}_{H^N}\le2B_N$, \eqref{eq:trap} yields
\[
\begin{aligned}
\norm{c(t)-c(T)}_{C^{2,\alpha}}
&\le C\norm{c(t)-c(T)}_{H^{-1}}^\beta
       \norm{c(t)-c(T)}_{H^N}^{1-\beta}\\
&\le C(2B_N)^{1-\beta}E(T)^{\beta\theta}
<\sigma/4.
\end{aligned}
\]
Passing to $t=S$ by continuity and using \eqref{ineq:CT<sigma/4}, we obtain
\[
\norm{c(S)-c^*}_{C^{2,\alpha}}
\le\norm{c(S)-c(T)}_{C^{2,\alpha}}
+\norm{c(T)-c^*}_{C^{2,\alpha}}
<\sigma/2,
\]
a contradiction. Hence $S=\infty$, and the trajectory remains in $\mathcal U$ for all $t\ge T$.

\par\smallskip
\noindent\emph{Step 3. Convergence of the whole trajectory.}
We may now apply \eqref{eq:length} to every interval $[T',S]$ with $S\ge T'\ge T$. Since $E(T')\to0$,
\[
\sup_{S\ge T'}\norm{c(S)-c(T')}_{H^{-1}}\le CE(T')^\theta\longrightarrow0.
\]
Completeness of $H^{-1}$ gives convergence. The known subsequence identifies its limit as $c^*$. For each integer $m\ge0$, choose $s>m+1$ and an integer $N>s$, and set $\beta=(N-s)/(N+1)\in(0,1)$. The two-dimensional Sobolev embedding and interpolation give
\[
\begin{aligned}
\norm{c(t)-c^*}_{C^m}
&\le C\norm{c(t)-c^*}_{H^s}\\
&\le C\norm{c(t)-c^*}_{H^{-1}}^\beta
       \norm{c(t)-c^*}_{H^N}^{1-\beta}\\
&\le C\norm{c(t)-c^*}_{H^{-1}}^\beta
\longrightarrow0,
\end{aligned}
\]
where the last inequality uses the uniform $H^N$ bound and the smoothness of $c^*$. Letting $S\to\infty$ in \eqref{eq:length} also proves finite $H^{-1}$ length.
\end{proof}

\begin{proof}[Proof of Theorem~\ref{thm:main}]
Lemma~\ref{lem:local} constructs the unique maximal positive solution. Propositions~\ref{prop:coercivity} and \ref{prop:Lp} supply the estimates used in Proposition~\ref{prop:regularity} to prove global existence and \eqref{eq:high}. Proposition~\ref{prop:subsequence} gives a positive stationary subsequential limit and Proposition~\ref{prop:whole} proves \eqref{eq:whole}. Lemma~\ref{lem:stationary} gives \eqref{eq:elliptic} with the zero-mean normalization.
\end{proof}

\section{Two complementary results}\label{s:Complementary}

The following theorem supplements Theorem~\ref{thm:main} with quantitative
convergence estimates.

\begin{theorem}\label{thm:rates}
Under the assumptions of Theorem~\ref{thm:main}, let $c(t)$ be the solution and let $c^*=c^\infty$ be its stationary limit. There exists an exponent $\theta\in(0,1/2]$ for which the following estimates hold. Set
\[
 E(t)=\mathcal F_\rho(c(t))-\mathcal F_\rho(c^*),
 \qquad u_i^*=U_i[c^*].
\]
There exist $T\geq0$ and $C,\kappa>0$ such that, for all $t\geq T$, we get
\[
 \|c(t)-c^*\|_{H^{-1}(M)^n}+d_\rho(c(t),c^*)
 \leq C E(t)^\theta.
\]
If $0<\theta<1/2$, then
\begin{equation}\label{eq:rates-algebraic}\begin{aligned}
 &E(t)\leq C(1+t-T)^{-1/(1-2\theta)},\\
 &\|c(t)-c^*\|_{H^{-1}(M)^n}+d_\rho(c(t),c^*)
 \leq C(1+t-T)^{-\theta/(1-2\theta)}.
\end{aligned}\end{equation}
If $\theta=1/2$, then
\begin{equation}\label{eq:rates-exponential}\begin{aligned}
 &E(t)\leq C e^{-\kappa(t-T)},\\
 &\|c(t)-c^*\|_{H^{-1}(M)^n}+d_\rho(c(t),c^*)
 \leq C e^{-\kappa(t-T)/2}.
\end{aligned}\end{equation}
Moreover, for every integer $m\geq0$, every $\alpha\in(0,1)$ and every $b\in(0,1)$, there exists a constant $C_{m,\alpha,b}$ such that
\begin{equation}\label{eq:rates-high}\begin{split}
 \|c(t)-c^*\|_{C^{m,\alpha}(M)^n}
 &+\sum_i\|U_i[c(t)]-u_i^*\|_{C^{m+2,\alpha}(M)}
 \\ &\leq C_{m,\alpha,b}
 \begin{cases}
 (1+t-T)^{-b\theta/(1-2\theta)},&0<\theta<1/2,\\
 e^{-b\kappa(t-T)/2},&\theta=1/2.
 \end{cases}
\end{split}\end{equation}
All constants depend on the fixed equation data, the trajectory and the limiting stationary density. The higher-norm constants also
depend on $m,\alpha,b$.
\end{theorem}

\begin{proof}
\par\smallskip
\noindent\emph{Step 1. Energy and $H^{-1}$ estimates.}
By Proposition~\ref{prop:whole}, the trajectory remains in the neighborhood of Lemma~\ref{lem:LS} for all sufficiently large $t\ge T$. The estimates established in Step~1 of its proof, integrated from $t$ to infinity, give
\[
E(t)^{1-\theta}\le C_0\D(c(t))^{1/2},
\qquad
\int_t^\infty\D(c(r))^{1/2}\,\mathrm dr
\le \frac{C_0}{\theta}E(t)^\theta,
\qquad
\|c(t)-c^*\|_{H^{-1}}\le CE(t)^\theta.
\]
Here the second estimate follows from the differential inequality for $E^\theta$, and the third follows from \eqref{eq:length} and convergence to $c^*$.

Set $\kappa=C_0^{-2}$. The energy identity implies
\[
E'(t)=-\D(c(t))\le-\kappa E(t)^{2(1-\theta)}.
\]
As in Proposition~\ref{prop:whole}, if $E$ reaches zero, the solution equals $c^*$ thereafter. Otherwise, integration gives
\[
E(t)\le
\begin{cases}
\bigl(E(T)^{-q}+q\kappa(t-T)\bigr)^{-1/q},
  &q=1-2\theta>0,\\
E(T)e^{-\kappa(t-T)},&\theta=1/2.
\end{cases}
\]
Together with the preceding $H^{-1}$ estimate, this proves the energy and $H^{-1}$ rates in \eqref{eq:rates-algebraic}--\eqref{eq:rates-exponential}.

\par\smallskip
\noindent\emph{Step 2. Transport distance.}
Put $p_i=c_i/\rho_i$ and $v_i=-\nabla_g\mu_i$, so that $\partial_t p_i+\Div(p_iv_i)=0$. On $[t,s]$, let $X_i(r;t,x)$ be the smooth flow of $v_i$. The characteristic representation for the continuity equation (cf.~\cite[Proposition~2.12]{AS}, whose proof applies in local coordinates on $M$) gives
\[
X_i(s;t,\cdot)_\#(p_i(t)\,dV_g)=p_i(s)\,dV_g.
\]
The induced couplings, the bound of endpoint distance by path length, and Minkowski's inequality therefore yield
\[
\begin{aligned}
d_\rho(c(t),c(s))
\le\int_t^s
\left(\sum_i\int_M c_i(r)|v_i(r)|_g^2\,dV_g\right)^{1/2}
\,\mathrm dr=\int_t^s\D(c(r))^{1/2}\,\mathrm dr.
\end{aligned}
\]
Coupling the common parts of the component measures along the diagonal and their remaining masses arbitrarily gives
\[
d_\rho(c(s),c^*)^2
\le\frac{\operatorname{diam}(M)^2}{2}
\sum_i\|c_i(s)-c_i^*\|_{L^1}\longrightarrow0.
\]
Letting $s\to\infty$ and using Step~1 proves $d_\rho(c(t),c^*)\le CE(t)^\theta$, hence the asserted transport rates.

\par\smallskip
\noindent\emph{Step 3. Higher spatial norms.}
Fix $m\ge0$, $\alpha\in(0,1)$ and $b\in(0,1)$. Choose $s>m+\alpha+1$ and an integer $N>s$ so large that $\beta=(N-s)/(N+1)>b$. Using the same interpolation as in Proposition~\ref{prop:whole}, together with the uniform $H^N$ bounds, gives
\begin{align}
\|c(t)-c^*\|_{C^{m,\alpha}}
\le C\|c(t)-c^*\|_{H^{-1}}^\beta
       \|c(t)-c^*\|_{H^N}^{1-\beta}\le C_{m,\alpha,b}E(t)^{\beta\theta}.
\end{align}
Moreover, \eqref{eq:poisson} and the elliptic Schauder estimate give
\[
\sum_i\|U_i[c(t)]-u_i^*\|_{C^{m+2,\alpha}}
\le C\|c(t)-c^*\|_{C^{m,\alpha}}.
\]
Substituting the energy rates and using $\beta>b$ proves \eqref{eq:rates-high}.
\end{proof}

\begin{remark}
When $0<\theta<1/2$, Theorem~\ref{thm:rates} gives decay in each fixed $C^{m,\alpha}$ norm at every algebraic exponent strictly smaller than $\theta/(1-2\theta)$. When $\theta=1/2$, it gives every exponential exponent strictly smaller than $\kappa/2$. The endpoint rates in these strong norms do not follow merely by sending $N$ to infinity, since the interpolation constants depend on $N$.
\end{remark}

\begin{remark}
If the constrained Hessian $L$ in \eqref{eq:Hessian} has trivial kernel, the proof of Lemma~\ref{lem:LS} gives $\theta=1/2$. Hence the exponential estimates in Theorem~\ref{thm:rates} apply. This is an additional condition on the limiting stationary density; Theorem~\ref{thm:main} also allows a nontrivial kernel.
\end{remark}

In follows, we show that the energy admits no lower bound if $\Lambda_J(\rho)<0$ for some nonempty $J\subset I$.

\begin{proposition}\label{prop:negative}
If $\Lambda_J(\rho)<0$ for some nonempty $J\subseteq I$, then $\F$ is unbounded below on the nonnegative fixed-mass class of finite entropy.
\end{proposition}
\begin{proof}
Choose normal coordinates at $p\in M$ and a nonnegative $\phi\in C_c^\infty(B_1(0))$ with $\int_{\mathbb R^2}\phi=1$. Let $J_g$ denote the volume density in these coordinates and set
\[
Z_\varepsilon=\int\phi(z)J_g(\varepsilon z)\,\mathrm dz=1+O(\varepsilon^2),
\qquad
c_i^\varepsilon(\exp_p(\varepsilon z))=
\frac{\rho_i}{Z_\varepsilon}\varepsilon^{-2}\phi(z),\quad i\in J.
\]
Extend these functions by zero outside of $B_1(0)$. Their supports lie strictly inside the coordinate ball, so the extensions are smooth and the change of variables shows that their masses are exactly $\rho_i$. Fix smooth nonnegative densities of the prescribed masses for $i\notin J$. Since $J_g(y)=1+O(|y|^2)$ and $\phi\log\phi$ is integrable,
\[
\int c_i^\varepsilon\log c_i^\varepsilon\dV
=2\rho_i\log(1/\varepsilon)+O(1)\mbox{ for } i\in J\mbox{ as }\varepsilon\to0.
\]
The coordinate distance is uniformly comparable to $\varepsilon|z-w|$. Thus
\[
\log d_g(\exp_p(\varepsilon z),\exp_p(\varepsilon w))
=\log\varepsilon+\log|z-w|+E_\varepsilon(z,w),\qquad |E_\varepsilon|\le C
\]
off the diagonal. The probability measures $\mathrm d\nu_\varepsilon=Z_\varepsilon^{-1}\phi(z)J_g(\varepsilon z)\,\mathrm dz$ have uniformly bounded densities supported in $B_1$. Local integrability of $|\log|z-w||$ in two dimensions therefore bounds its double integral uniformly. To estimate the interaction energy, \eqref{eq:singularity} gives
\[
\int(c_i^\varepsilon-\rho_i)\G(c_j^\varepsilon-\rho_j)\dV
=\frac{\rho_i\rho_j}{2\pi}\log(1/\varepsilon)+O(1)\text{ for }i,j\in J\text{ as }\varepsilon\to0.
\]
If $j\notin J$, its fixed density produces a bounded smooth potential $V_j=\G(c_j-\rho_j)$ of zero mean and $|\int(c_i^\varepsilon-\rho_i)V_j\dV|\le\rho_i\norm{V_j}_\infty$. Symmetry handles the reversed mixed terms. The terms with both indices outside $J$ are fixed, and the external-potential energy remains bounded. Consequently,
\[
\F(c^\varepsilon)=\left(2\sum_{i\in J}\rho_i-
\frac1{4\pi}\sum_{i,j\in J}a_{ij}\rho_i\rho_j\right)\log(1/\varepsilon)+O(1)
=\frac{\Lambda_J(\rho)}{4\pi}\log(1/\varepsilon)+O(1)\longrightarrow-\infty
\]
as $\varepsilon\to0$.
\end{proof}

\begin{example}
Let
\[
A=\begin{pmatrix}1&1/10\\1/10&1\end{pmatrix},\qquad
A^{-1}=\frac{100}{99}\begin{pmatrix}1&-1/10\\-1/10&1\end{pmatrix}.
\]
The matrix is symmetric and entrywise nonnegative, but $a^{11}=a^{22}=100/99>0$ and $a^{12}=a^{21}=-10/99<0$, so $(H_2)$ fails. Taking $\rho_1=\rho_2=2\pi$, we obtain
\[
\frac{\Lambda_{\{1\}}(\rho)}{64\pi^2}
=\frac{\Lambda_{\{2\}}(\rho)}{64\pi^2}=\frac3{16}>0,
\qquad
\frac{\Lambda_{\{1,2\}}(\rho)}{64\pi^2}=\frac{29}{80}>0.
\]
Hence Theorem~\ref{thm:main} applies although $(H_2)$ does not hold. For the same matrix, the different mass vector $\rho=8\pi(11/10,2/5)$ gives
\[
\frac{\Lambda_I(\rho)}{64\pi^2}=\frac{21}{500}>0,
\qquad
\frac{\Lambda_{\{1\}}(\rho)}{64\pi^2}=-\frac{11}{100}<0.
\]
This also shows why the full-set inequality alone is insufficient for the general matrix class considered here.
\end{example}

Proposition~\ref{prop:negative} concerns failure of a lower energy bound. It does not by itself prove blow-up of the flow. The cases $\Lambda_J=0$ require a separate critical analysis. A stationary density is a constrained critical point of $\F$. Convergence to it does not assert that it minimizes $\F$.

\section*{Acknowledgements}
The authors are listed in alphabetical order of their surnames, and all authors have made equal contributions. The authors acknowledge the use of AI tools. All mathematical arguments and proofs in the final manuscript were checked and written by the authors.

\Needspace{17\baselineskip}

\end{document}